\documentclass[psamsfonts]{amsart}

\usepackage{amssymb,amsfonts}
\usepackage[all,arc]{xy}
\usepackage{enumerate}
\usepackage{mathrsfs}
\usepackage{textcomp}
\usepackage{cite}

\newtheorem{thm}{Theorem}[section]
\newtheorem{cor}[thm]{Corollary}
\newtheorem{prop}[thm]{Proposition}
\newtheorem{lem}[thm]{Lemma}

\newtheorem{quest}[thm]{Question}

\theoremstyle{definition}
\newtheorem{defn}[thm]{Definition}

\newtheorem{fact}[thm]{Fact}

\theoremstyle{remark}
\newtheorem{rem}[thm]{Remark}

\makeatletter
\let\c@equation\c@thm
\makeatother
\numberwithin{equation}{section}

\def\Ind{\setbox0=\hbox{$x$}\kern\wd0\hbox to 0pt{\hss$\mid$\hss} \lower.9\ht0\hbox to 0pt{\hss$\smile$\hss}\kern\wd0} 

\def\Notind{\setbox0=\hbox{$x$}\kern\wd0\hbox to 0pt{\mathchardef \nn=12854\hss$\nn$\kern1.4\wd0\hss}\hbox to 0pt{\hss$\mid$\hss}\lower.9\ht0 \hbox to 0pt{\hss$\smile$\hss}\kern\wd0} 

\def\ind{\mathop{\mathpalette\Ind{}}} 

\def\nind{\mathop{\mathpalette\Notind{}}}

\title{Transitivity, lowness, and ranks in NSOP$_{1}$ theories}

\author{Artem Chernikov, Byunghan Kim, Nicholas Ramsey}

\date{\today}

\thanks{Chernikov was partially supported by the NSF CAREER grant DMS-1651321 and by a Simons fellowship.  
Kim has been supported by  Samsung Science Technology Foundation under Project Number SSTF-BA1301-03 and NRF of Korea grants 2018R1D1A1A02085584 and 2021R1A2C1009639.}

\begin{document}

\begin{abstract}
We develop the theory of Kim-independence in the context of NSOP$_{1}$ theories satsifying the existence axiom.  We show that, in such theories, Kim-independence is transitive and that $\ind^{K}$-Morley sequences witness Kim-dividing.  As applications, we show that, under the assumption of existence, in a low NSOP$_{1}$ theory, Shelah strong types and Lascar strong types coincide and, additionally, we introduce a notion of rank for NSOP$_{1}$ theories.  
\end{abstract}

\maketitle

\setcounter{tocdepth}{1}
\tableofcontents

This paper furthers the development of the theory of Kim-independence in the context of NSOP$_{1}$ theories satisfying the existence axiom.  Building on earlier work in \cite{ArtemNick} and a suggestion of the second-named author \cite{KimNTP1}, Kim-independence was introduced in \cite{kaplan2017kim}, where it was shown to be a well-behaved notion of independence in NSOP$_{1}$ theories.  This work established a strong analogy between the theory of non-forking independence in simple theories and Kim-independence in NSOP$_{1}$ theories, an analogy which subsequent works have only deepened \cite{kaplan2019local, kaplan2019transitivity, kim2019number}.  Nonetheless, one major difference between the two notions of independence is that, unlike non-forking which makes sense over all sets, Kim-independence is only a sensible notion of independence \emph{over models}:  Kim-independence is defined in terms of formulas that divide with respect to a Morley sequence in a global invariant type, and such a sequence, in general, is only guaranteed to exist over a model.  In \cite{dobrowolski2019independence}, the second and third-named author, together with Dobrowolski, focused on the context of NSOP$_{1}$ theories that satisfy \emph{the existence} axiom.  There, it was shown that Kim-independence may be defined over arbitrary sets and basic theorems of Kim-independence over models hold in this broader context.  

The existence axiom states that every complete type has a global non-forking extension, i.e.~every set is an extension base in the terminology of \cite{chernikov2012forking}.  This is equivalent to the statement that, in every type, there is a (non-forking) Morley sequence and, hence, assuming existence, one may redefine Kim-independence in terms of the formulas that divide along Morley sequences of this kind.  New technical challenges arise in this setting, but in \cite{dobrowolski2019independence} it was shown that Kim-independence satisfies Kim's lemma, symmetry, and the independence theorem for Lascar strong types.  Moreover, all simple theories and all known examples in the growing list of NSOP$_{1}$ theories satisfy existence, and it is expected to hold in all NSOP$_{1}$ theories, see, e.g., \cite[Fact 2.14]{dobrowolski2019independence}.

Here we continue work on Kim-independence in NSOP$_{1}$ theories satisfying existence, in particular, exploring aspects of the theory that are too cumbersome or uninteresting over models.  In Section \ref{transitivity section}, we show that Kim-independence is transitive in an NSOP$_{1}$ theory satisfying existence and that, moreover, Kim-dividing is witnessed by $\ind^{K}$-Morley sequences.  These results were first established over models for all NSOP$_{1}$ theories in \cite{kaplan2019transitivity} and our proofs largely follow the same strategy.  Nonetheless, suitable replacements need to be found for notions that only make sense, in general, over models, like heirs and coheirs.  We find that arguments involving these notions can often be replaced by an argument involving a tree-induction, as in the construction of tree Morley sequences in \cite{kaplan2019transitivity}.  In Section \ref{low section}, we apply these results to low NSOP$_{1}$ theories satisfying existence, showing that Shelah strong types and Lascar strong types coincide, generalizing a result of Buechler for simple theories \cite{buechler1999lascar} (see also \cite{shami2000definability, kim2013simplicity}).  In Section \ref{rank section}, we introduce a notion of rank for NSOP$_{1}$ theories and establish some of its basic properties.  Finally, in Section \ref{Kim-Pillay theorem}, we generalize the Kim-Pillay criterion for Kim-independence from \cite[Theorem 6.1]{ArtemNick} and \cite[Theorem 6.11]{kaplan2019transitivity} to give a criterion for NSOP$_{1}$ in theories satisfying existence, which, additionally, gives an abstract characterization of Kim-independence over arbitrary sets in this setting.  


\section{Preliminaries}

In this paper, $T$ will always be a complete theory with monster model $\mathbb{M}$.  We will implicitly assume all models and sets of parameters are small, that is, of cardinality less than the degree of saturation and homogeneity of $\mathbb{M}$.  If we discuss an $I$-indexed indiscernible sequence $(a_{i})_{i \in I}$, we will implicitly assume $I$ is linearly ordered by $<$ and, given $i \in I$, we will write $a_{<i}$ and $a_{\leq i}$ for the subsequences $(a_{j})_{j < i}$ and $(a_{j})_{j \leq i}$ respectively. 

\begin{defn}
Suppose $A$ is a set of parameters.
\begin{enumerate}
\item We say that a formula $\varphi(x;a)$ \emph{divides} over a set $A$ if there is an $A$-indiscernible sequence $\langle a_{i} : i < \omega \rangle$ with $a_{0} = a$ such that $\{\varphi(x;a_{i}) : i < \omega\}$ is inconsistent.  
\item A formula $\varphi(x;a)$ is said to \emph{fork} over $A$ if $\varphi(x;a) \vdash \bigvee_{i < k} \psi_{i}(x;c_{i})$, for some $k < \omega$, with $\psi_{i}(x;c_{i})$ dividing over $A$.  
\item We say a partial type divides (forks) over $A$ if it implies a formula that divides (forks) over $A$.  
\item For tuples $a$ and $b$, we write $a \ind^{d}_{A} b$ or $a \ind_{A} b$ to indicate that $\text{tp}(a/Ab)$ does not divide over $A$ or does not fork over $A$, respectively.  
\item A \emph{Morley sequence} $(a_{i})_{i \in I}$ over $A$ is an infinite $A$-indiscernible sequence such that $a_{i} \ind_{A} a_{<i}$ for all $i \in I$.  If $p \in S(A)$, we say $(a_{i})_{i \in I}$ is a Morley sequence \emph{in p} if, additionally, $a_{i} \models p$ for all $i \in I$.  
\end{enumerate}
\end{defn}

The following is one of the key definitions of this paper.  It defines a context in which Kim-independence may be studied over arbitrary sets.

\begin{defn}
We define the \emph{existence axiom} to be any one of the following equivalent conditions on $T$:
\begin{enumerate}
\item For all parameter sets $A$, any type $p \in S(A)$ does not fork over $A$.  
\item For all parameter sets $A$, no consistent formula over $A$ forks over $A$.  
\item For all parameter sets $A$, every type $p \in S(A)$ has a global extension that does not fork over $A$.
\item For all parameter sets $A$ and any $p\in S(A)$, there is a Morley sequence in $p$.
\end{enumerate}
If $T$ satisfies the existence axiom, we will often abbreviate this by writing $T$ is \emph{with existence}.  See, e.g., \cite[Remark 2.6]{dobrowolski2019independence} for the equivalence of (1)\textemdash (4).  
\end{defn}

Under existence, we may define Kim-independence over arbitrary sets.  The following definition was given in \cite{dobrowolski2019independence}, but it was observed already in \cite[Theorem 7.7]{kaplan2017kim} that this agrees with the original definition over models.  

\begin{defn}
Suppose $T$ satisfies the existence axiom.
\begin{enumerate}
\item We say a formula $\varphi(x;a)$ \emph{Kim-divides} over $A$ if there is a sequence $\langle a_{i} : i < \omega\rangle$ which is a Morley sequence over $A$ with $a_{0} = a$ and $\{\varphi(x;a_{i}) : i < \omega\}$ inconsistent.
\item A formula $\varphi(x;a)$ is said to \emph{Kim-fork} over $A$ if $\varphi(x;a) \vdash \bigvee_{i < k} \psi_{i}(x;c_{i})$, where each $\psi_{i}(x;c_{i})$ Kim-divides over $A$.  
\item We say a type Kim-divides (Kim-forks) over $A$ if it implies a formula that Kim-divides (Kim-forks) over $A$.  
\item For tuples $a$ and $b$, we write $a \ind^{K}_{A} b$ to indicate that $\text{tp}(a/Ab)$ does not Kim-divide over $A$.  
\item An $\ind^{K}$-\emph{Morley sequence} $(a_{i})_{i \in I}$ over $A$ is an infinite $A$-indiscernible sequence such that $a_{i} \ind^{K}_{A} a_{<i}$ for all $i \in I$.  
\end{enumerate}
\end{defn}

\begin{rem}
By Kim's lemma \cite[Proposition 2.2.6]{kim2013simplicity}, if $T$ is simple, a formula Kim-divides over a set $A$ if and only if it divides over $A$.  
\end{rem}

\begin{defn} \cite[Definition 2.2]{dvzamonja2004maximality} \label{sop1def}
The formula $\varphi(x;y)$ has SOP$_{1}$ if there is a collection of tuples $(a_{\eta})_{\eta \in 2^{<\omega}}$ so that 
\begin{itemize}
\item For all $\eta \in 2^{\omega}$, $\{\varphi(x;a_{\eta | \alpha}) : \alpha < \omega\}$ is consistent.
\item For all $\eta \in 2^{<\omega}$, if $\nu$ extends $\eta \frown \langle 0 \rangle$, then $\{\varphi(x;a_{\nu}), \varphi(x;a_{\eta \frown 1})\}$ is inconsistent,
\end{itemize}
where $\unlhd$ denotes the tree partial order on $2^{<\omega}$.  We say $T$ is SOP$_{1}$ if some formula has SOP$_{1}$ modulo $T$.  $T$ is NSOP$_{1}$ otherwise.  
\end{defn}

\begin{defn}
Suppose $A$ is a set of parameters.  
\begin{enumerate}
\item We say that tuples $a$ and $b$ \emph{have the same (Shelah) strong type over} $A$, written $a \equiv^{S}_{A} b$, if $E(a,b)$ holds (i.e.~$E(a',b')$ holds for all corresponding finite subtuples $a'$ and $b'$ of $a$ and $b$ respectively) for every $A$-definable equivalence relation $E(x,y)$ with finitely many classes.  
\item The group $\mathrm{Autf}(\mathbb{M}/A)$ of \emph{Lascar strong automorphisms (of the monster) over }$A$ is the subgroup of $\mathrm{Aut}(\mathbb{M}/A)$ generated by $\bigcup \{\mathrm{Aut}(\mathbb{M}/M) : A \subseteq M \prec \mathbb{M} \}$.  We say $a$ and $b$ have \emph{the same Lascar strong type over }$A$, written $a \equiv^{L}_{A} b$, if there is some $\sigma \in \mathrm{Autf}(\mathbb{M}/A)$ such that $\sigma(a) = b$.  By a \emph{Lascar strong type} over $A$, we mean an equivalence class of the relation $\equiv^{L}_{A}$.  
\item A type-definable equivalence relation $E$ on $\alpha$-tuples, for an ordinal $\alpha$, is called \emph{bounded} if it has small  number of classes.  We say $a$ and $b$ \emph{have the same }KP\emph{-strong type over }$A$, written $a \equiv^{\mathrm{KP}}_{A} b$, if $E(a,b)$ holds for all bounded type-definable equivalence relations over $A$. 
\item We say that $T$ is $G$\emph{-compact over} $A$ when $a \equiv^{L}_{A} b$ if and only if $a \equiv^{\mathrm{KP}}_{A} b$ for all (possibly infinite) tuples $a,b$.  We say $T$ is $G$\emph{-compact} if it is $G$-compact over all finite sets $A$.  
\end{enumerate} 
\end{defn}

%
In \cite{dobrowolski2019independence}, several basic facts about Kim-independence in NSOP$_{1}$ theories with existence were established.  As we will make extensive use of them throughout the paper, we record them below.

\begin{fact} \label{basic Kim-independence facts}
Assume $T$ is NSOP$_{1}$ with existence and $A$ is a set of parameters.  Then the following properties hold.  
\begin{enumerate}
\item Extension: If $\pi(x)$ is a partial type over $B \supseteq A$ which does not Kim-divide over $A$, then there is a completion $p \in S(B)$ of $\pi$ that does not Kim-divide over $A$.  In particular, if $a \ind^{K}_{A} b$ and $c$ is arbitrary, there is some $a' \equiv_{Ab} a$ such that $a \ind^{K}_{A} bc$.  \cite[Proposition 4.1]{dobrowolski2019independence}
\item Symmetry:  $a \ind^{K}_{A} b \iff b \ind^{K}_{A} a$. \cite[Corollary 4.9]{dobrowolski2019independence}
\item Kim's Lemma for Morley sequences:  the formula $\varphi(x;a)$ Kim-divides over $A$ if and only if $\{\varphi(x;a_{i}) : i < \omega\}$ is inconsistent for all Morley sequences $\langle a_{i} : i < \omega\rangle$ over $A$ with $a_{0} = a$.  \cite[Theorem 3.5]{dobrowolski2019independence}
\item Kim-forking = Kim-dividing: if a formula $\varphi(x;a)$ Kim-forks over $A$, then $\varphi(x;a)$ Kim-divides over $A$.  \cite[Proposition 4.1]{dobrowolski2019independence}
\item The chain condition:  if $a \ind^{K}_{A} b$ and $I = \langle b_{i} : i < \omega \rangle$ is a Morley sequence over $A$ with $b_{0} = b$, then there is $a' \equiv_{Ab} a$ such that $I$ is $Aa'$-indiscernible and $a' \ind^{K}_{A} I$ (this follows from (3), as in, e.g., \cite[Corollary 5.15]{dobrowolski2019independence}).  
\item The independence theorem for Lascar strong types: if $a_{0} \equiv^{L}_{A} a_{1}$, $a_{0} \ind^{K}_{A} b$, $a_{1} \ind^{K}_{A} c$, and $b \ind^{K}_{A} c$, then there is some $a_{*}$ with $a_{*} \equiv^{L}_{Ab} a_{0}$, $a_{*} \equiv^{L}_{Ac} a_{1}$, and $a_{*} \ind^{K}_{A} bc$.  \cite[Theorem 5.8]{dobrowolski2019independence}
\item $T_{A}$ is $G$-compact for any small set $A$, where $T_{A}$ is the theory of the monster model in the language with constants for the elements of $A$.  \cite[Corollary 5.9]{dobrowolski2019independence}
\end{enumerate}
\end{fact}

As these facts make up part of the standard tool box for reasoning about Kim-independence, we will often make implicit use of these properties. For example, Kim's Lemma for Morley sequences, Item (3) in the above list, is often used in this paper in the following way:  if $I = \langle a_{i} : i < \omega \rangle$ is a Morley sequence over $A$ with $a_{0} = a$ which is $Ab$-indiscernible, then $a \ind^{K}_{A} b$.  To see this, by symmetry (Item (2)), it suffices to show that $b \ind^{K}_{A} a$ which, by Item (4), means that we need to show that there is no formula  $\varphi(x;a) \in \text{tp}(b/Aa)$ which Kim-divides over $A$.  But if $\varphi(x;a)$ Kim-divides over $A$, then Kim's lemma implies that $\{\varphi(x;a_{i}) : i < \omega\}$ is inconsistent.  This set of formulas, however, is realized by $b$ so there can be no such formula.  

The following is local character of Kim-independence for NSOP$_{1}$ theories.  The usual formulation of local character for non-forking independence in simple theories merely asserts that, for any type $p \in S(A)$, the set of $B$ with $|B| \leq |T|$ such that $p$ does not fork over $B$ is \emph{non-empty}, but it follows by base monotonicity, then, that $p$ does not fork over $C$ for any $B \subseteq C \subseteq A$.  Because Kim-independence, in general, does not satisfy base monotonicity in NSOP$_{1}$ theories, the following is the appropriate analogue for this setting:  

\begin{fact} \label{local character} \cite[Theorem 3.9]{kaplan2019local}
Suppose $T$ is NSOP$_{1}$ and $M \models T$ with $|M| \geq |T|$.  Given any $p \in S(M)$ (in finitely many variables), the set $X$ defined by 
$$
X := \{N \prec M : |N| = |T| \text{ and } p \text{ does not Kim-divide over }N\}
$$
satisfies the following:
\begin{enumerate}
\item $X$ is closed:  if $\langle N_{i} : i < |T| \rangle$ is a sequence of models in $X$ with $N_{i} \subseteq N_{j}$ for all $i < j$, then $\bigcup_{i < |T|} N_{i} \in X$.  
\item $X$ is unbounded:  if $Y \subset M$ has cardinality $\leq |T|$, there is some $N \in X$ with $Y \subseteq N$.  
\end{enumerate}
\end{fact}

\begin{rem}
It is an easy consequence of Fact \ref{local character} that if $M \models T$ is equal to the union of $\langle N_{i} : i < |T|^{+}\rangle$, an increasing and continuous (i.e.~$N_{\delta} = \bigcup_{i < \delta} N_{i}$ for all limit $\delta$) elementary chain of models of $T$ of size $|T|$, then for any $p \in S(M)$, there is some $i < |T|^{+}$ such that $p$ does not Kim-divide over $N_{i}$. 
\end{rem}

\subsection{Trees}

At several points in the paper, we will construct indiscernible sequences by an inductive construction of indiscernible trees.  We recall the basic framework for these `tree-inductions' from \cite{kaplan2017kim}.  For an ordinal $\alpha$, let the language $L_{s,\alpha}$ be $\langle \unlhd, \wedge, <_{lex}, (P_{\beta})_{\beta \leq \alpha} \rangle$.  For us, a tree will mean a partial order $\unlhd$ such that for all $x$, the elements $\{y : y \unlhd x\}$ below $x$ are linearly ordered (and not necessarily well-ordered) by $\unlhd$ and such that for all $x,y$, $x$ and $y$ have an infimum, i.e. there is a $\unlhd$-greatest element $z \unlhd x,y$, which is called the meet of $x$ and $y$.  We may view a tree with $\alpha$ levels as an $L_{s,\alpha}$-structure by interpreting $\unlhd$ as the tree partial order, $\wedge$ as the binary meet function, $<_{lex}$ as the lexicographic order, and $P_{\beta}$ interpreted to define level $\beta$. The specific trees, and the interpretations of these symbols that turn them into $L_{s,\alpha}$-structures, that we will need in this paper are outlined precisely in Definition \ref{tee alphas} below. 

We now recall the modeling property.  In what follows, we will write $\mathrm{qftp}_{L'}(a)$ to denote the quantifier-free type of $a$ in the language $L'$ and write $\text{tp}_{\Delta}(b/A)$ to denote the $\Delta$-type of $b$ over $A$ (i.e. the set of positive and negative instances of formulas in $\Delta$ with parameters from $A$ satisfied by $b$).  Although the subscript is used in two conflicting ways, it will be clear from context which is intended. 

\begin{defn}  Suppose $I$ is an $L'$-structure, where $L'$ is some language. 
\begin{enumerate}
\item  We say $(a_{i} : i \in I)$ is a set of $I$\emph{-indexed indiscernibles over $A$} if whenever 

$(s_{0}, \ldots, s_{n-1})$, $(t_{0}, \ldots, t_{n-1})$ are tuples from $I$ with 
$$
\text{qftp}_{L'}(s_{0}, \ldots, s_{n-1}) = \text{qftp}_{L'}(t_{0}, \ldots, t_{n-1}),
$$
then we have
$$
\text{tp}(a_{s_{0}},\ldots, a_{s_{n-1}}/A) = \text{tp}(a_{t_{0}},\ldots, a_{t_{n-1}}/A).
$$
\item In the case that $L' = L_{s,\alpha}$ for some $\alpha$, we say that an $I$-indexed indiscernible is $\emph{s-indiscernible}$.  As the only $L_{s,\alpha}$-structures we will consider will be trees, we will often refer to $I$-indexed indiscernibles in this case as \emph{s-indiscernible trees}.  
\item We say that $I$-indexed indiscernibles have the \emph{modeling property} if, given any $(a_{i} : i \in I)$ from $\mathbb{M}$ and any $A$, there is an \(I\)-indexed indiscernible \((b_{i} : i \in I)\) over $A$ in $\mathbb{M}$ \emph{locally based} on \((a_{i} : i \in I)$ over $A$.  That is, given any finite set of formulas \(\Delta\) from $L(A)$ and a finite tuple \((t_{0}, \ldots, t_{n-1})\) from \(I\), there is a tuple \((s_{0}, \ldots, s_{n-1})\) from \(I\) so that 
\[
\text{qftp}_{L'} (t_{0}, \ldots, t_{n-1}) =\text{qftp}_{L'}(s_{0}, \ldots , s_{n-1})
\]
and also 
\[
\text{tp}_{\Delta}(b_{t_{0}}, \ldots, b_{t_{n-1}}) = \text{tp}_{\Delta}(a_{s_{0}}, \ldots, a_{s_{n-1}}).
\]
\end{enumerate}
\end{defn}

\begin{fact}\cite[Theorem 4.3]{KimKimScow}\label{modeling}
Let  \(I_{s}\) denote the \(L_{s,\omega}\)-structure 
$$
I_{s} = (\omega^{<\omega}, \unlhd, <_{lex}, \wedge, (P_{\alpha})_{\alpha < \omega})
$$
with all symbols being given their intended interpretations and each \(P_{\alpha}\) naming the elements of the tree at level \(\alpha\).  Then \(I_{s}\)-indexed indiscernibles have the modeling property.  
\end{fact}

Our trees will be understood to be an $L_{s,\alpha}$-structure for some appropriate $\alpha$.  As in \cite{kaplan2017kim}, we introduce a distinguished class of trees $\mathcal{T}_{\alpha}$.

\begin{defn} \label{tee alphas}
Suppose $\alpha$ is an ordinal.  We define $\mathcal{T}_{\alpha}$ to be the set of functions $f$ such that 
\begin{itemize}
\item $\text{dom}(f)$ is an end-segment of $\alpha$ of the form $[\beta,\alpha)$ for $\beta$ equal to $0$ or a successor ordinal.  If $\alpha$ is a successor, we allow $\beta = \alpha$, i.e.~$\text{dom}(f) = \emptyset$.
\item $\text{ran}(f) \subseteq \omega$.
\item finite support:  the set $\{\gamma \in \text{dom}(f) : f(\gamma) \neq 0\}$ is finite.    
\end{itemize}
We interpret $\mathcal{T}_{\alpha}$ as an $L_{s,\alpha}$-structure by defining 
\begin{itemize}
\item $f \unlhd g$ if and only if $f \subseteq g$.  Write $f \perp g$ if $\neg(f \unlhd g)$ and $\neg(g \unlhd f)$.  
\item $f \wedge g = f|_{[\beta, \alpha)} = g|_{[\beta, \alpha)}$ where $\beta = \text{min}\{ \gamma : f|_{[\gamma, \alpha)} =g|_{[\gamma, \alpha)}\}$, if non-empty (note that $\beta$ will not be a limit, by finite support). Define $f \wedge g$ to be the empty function if this set is empty (note that this cannot occur if $\alpha$ is a limit).  
\item $f <_{lex} g$ if and only if $f \vartriangleleft g$ or, $f \perp g$ with $\text{dom}(f \wedge g) = [\gamma +1,\alpha)$ and $f(\gamma) < g(\gamma)$
\item For all $\beta \leq \alpha$, $P_{\beta} = \{ f \in \mathcal{T}_{\alpha} : \text{dom}(f) = [\beta, \alpha)\}$.  Note that $P_{0}$ are the leaves of the tree (i.e.~the \emph{top} level) and $P_{\alpha}$ is empty for $\alpha$ limit.  
\end{itemize}
\end{defn}

Fact \ref{modeling} and compactness can be used to show that $\mathcal{T}_{\alpha}$-indexed indiscernibles have the modeling property as well \cite[Corollary 5.6]{kaplan2017kim}. 

\begin{defn}
Suppose $\alpha$ is an ordinal.  
\begin{enumerate}
\item (Restriction) If $v \subseteq \alpha$, the \emph{restriction of} $\mathcal{T}_{\alpha}$ \emph{to the set of levels  }$v$ is the $L_{s,\alpha}$-substructure of $\mathcal{T}_{\alpha}$ with the following underlying set:
$$
\mathcal{T}_{\alpha} \upharpoonright v = \{\eta \in \mathcal{T}_{\alpha} : \min (\text{dom}(\eta)) \in v \text{ and }\beta \in \text{dom}(\eta) \setminus v \implies \eta(\beta) = 0\}.
$$
\item (Concatenation)  If $\eta \in \mathcal{T}_{\alpha}$, $\text{dom}(\eta) = [\beta+1,\alpha)$ for $\beta$ non-limit, and $i < \omega$, let $\eta \frown \langle i \rangle$ denote the function $\eta \cup \{(\beta,i)\}$.  We define $\langle i \rangle \frown \eta \in \mathcal{T}_{\alpha+1}$ to be $\eta \cup \{(\alpha,i)\}$.  We write $\langle i \rangle$ for $\emptyset \frown \langle i \rangle$.
\item (Canonical inclusions) If $\alpha < \beta$, we define the map $\iota_{\alpha \beta} : \mathcal{T}_{\alpha} \to \mathcal{T}_{\beta}$ by $\iota_{\alpha \beta}(f) := f \cup \{(\gamma, 0) : \gamma \in \beta \setminus \alpha\}$.
\item (The all $0$'s path) If $\beta < \alpha$, then $\zeta_{\beta}$ denotes the function with $\text{dom}(\zeta_{\beta}) = [\beta, \alpha)$ and $\zeta_{\beta}(\gamma) = 0$ for all $\gamma \in [\beta,\alpha)$.  This defines an element of $\mathcal{T}_{\alpha}$ if and only if $\beta \in \{\gamma \in\alpha\mid  \gamma \mbox{ is not limit}\}=:[\alpha]$.  
\item (Tuple notation) Given $\nu \in \mathcal{T}_{\alpha}$, we write $a_{\unrhd \nu}$ for the tuple enumerating $\{a_{\xi} : \nu \unlhd \xi \in \mathcal{T}_{\alpha}\}$.
\end{enumerate}
\end{defn}

In previous works on Kim-independence over arbitrary sets, there was a gap concerning the construction of Morley trees (and a parallel gap in the theory over models), first discovered by Jan Dobrowolski and Mark Kamsma.  Namely, there is no reason \emph{a priori} for an $s$-indiscernible tree locally based on a weakly spread out tree (see Definition \ref{weaklyspread}) to be weakly spread out, which is needed to continue the induction. Over models this has a very easy fix:  one can check that it is possible to choose a global $M$-invariant type $q \supseteq \text{tp}((a_{\eta})_{\eta \in \mathcal{T}_{\alpha}}/M)$ such that, if $(a'_{\eta})_{\eta \in \mathcal{T}_{\alpha}} \models q$ then $(a'_{\eta})_{\eta \in \mathcal{T}_{\alpha}}$ is $\mathbb{M}$-indiscernible.  Morley sequences in such types allow the argument to work without change (the details for this case will appear elsewhere).  But over sets, a lengthier argument is required to patch the proofs. The relevant notion for the modification is that of a \emph{mutually s-indiscernible sequence}.  We prove in Lemma \ref{Morley mutual} that, given an $s$-indiscernible tree, there is a Morley sequence starting with this tree which is mutually s-indiscernible, and then we show in Lemma \ref{mutual preservation} that this notion is preserved upon passage to an $s$-indiscernible tree. 

\begin{defn}
We say a sequence $\langle (a_{\eta,i})_{\eta \in \mathcal{T}_{\alpha}} : i < \kappa \rangle$ is \emph{mutually} $s$\emph{-indiscernible over} $A$ if, for all $i < \kappa$, $(a_{\eta,i})_{\eta \in \mathcal{T}_{\alpha}}$ is $s$-indiscernible over $A\{a_{\eta,j} : \eta \in \mathcal{T}_{\alpha}, j \neq i \}$. 
\end{defn}

\begin{lem} \label{Morley mutual} 
Assume $A$ is an extension base. Given a tree $(a_{\eta})_{\eta \in \mathcal{T}_{\alpha}}$ that is $s$-indiscernible over $A$, there is a sequence $I = \langle (a_{\eta,i})_{\eta \in \mathcal{T}_{\alpha}} : i < \omega \rangle$ such that $(a_{\eta,0})_{\eta \in \mathcal{T}_{\alpha}} = (a_{\eta})_{\eta \in \mathcal{T}_{\alpha}}$, $I$ is a Morley sequence over $A$, and $I$ is mutually $s$-indiscernible over $A$.
\end{lem}

\begin{proof}
Let $\kappa$ be sufficiently large with respect to $|A|$.  By induction on $\gamma < \kappa$, we will choose $(a_{\eta,\gamma})_{\eta \in \mathcal{T}_{\alpha}}$ such that, taking $I_{\gamma}$ to be the sequence $\langle (a_{\eta,i})_{\eta \in \mathcal{T}_{\alpha}} : i < \gamma \rangle$, we have that $I_{\gamma}$ starts with $(a_{\eta})_{\eta \in \mathcal{T}_{\alpha}}$, is mutually $s$-indiscernible over $A$, and satisfies 
$$
(a_{\eta,i})_{\eta \in \mathcal{T}_{\alpha}} \ind_{A} (a_{\eta,j})_{\eta \in \mathcal{T}_{\alpha}, j < i}
$$
for all $i < \gamma$. The sequence $I_{1}$ is already specified and trivially satisfies the requirements.  

Assume we are given $(a_{\eta,i})_{\eta \in \mathcal{T}_{\alpha}}$ for all $i < \gamma$ and set $I_{\gamma} = \langle (a_{\eta,i})_{\eta \in \mathcal{T}_{\alpha}} : i < \gamma \rangle$.  Apply extension to get some $(b_{\eta})_{\eta \in \mathcal{T}_{\alpha}} \equiv_{A} (a_{\eta})_{\eta \in \mathcal{T}_{\alpha}}$ such that 
$$
(b_{\eta})_{\eta \in \mathcal{T}_{\alpha}} \ind_{A} I_{\gamma}. 
$$
By the modeling property, we can take $(b'_{\eta})_{\eta \in \mathcal{T}_{\alpha}}$ to be locally based on $(b_{\eta})_{\eta \in \mathcal{T}_{\alpha}}$ and $s$-indiscernible over $AI_{\gamma}$, then we we still have $(b'_{\eta})_{\eta \in \mathcal{T}_{\alpha}} \equiv_{A} (a_{\eta})_{\eta \in \mathcal{T}_{\alpha}}$, as $(a_{\eta})_{\eta \in \mathcal{T}_{\alpha}}$ was assumed to be $s$-indiscernible over $A$, and local basedness and strong finite character of non-forking imply 
$$
(b'_{\eta})_{\eta \in \mathcal{T}_{\alpha}} \ind_{A} I_{\gamma}.
$$ 
Now by induction on $i < \gamma$, we will choose $(a'_{\eta,i})_{\eta \in \mathcal{T}_{\alpha}}$ satisfying the following conditions:
\begin{enumerate}
    \item $(a'_{\eta,i})_{\eta \in \mathcal{T}_{\alpha}}$ is $s$-indiscernible over
    $$
    A \cup \{a'_{\eta,j} : \eta \in \mathcal{T}_{\alpha}, j < i\} \cup \{a_{\eta,k} : \eta \in \mathcal{T}_{\alpha}, k > i\} \cup \{b'_{\eta} : \eta \in \mathcal{T}_{\alpha}\}.
    $$
    \item $(b'_{\eta})_{\eta \in \mathcal{T}_{\alpha}}$ is $s$-indiscernible over 
    $$
    A \cup \{a'_{\eta,j} : \eta \in \mathcal{T}_{\alpha}, j \leq i\} \cup \{a_{\eta,k} : \eta \in \mathcal{T}_{\alpha}, k > i\}.
    $$
    \item $(a'_{\eta,j})_{\eta \in \mathcal{T}_{\alpha}, j \leq i} (a_{\eta,k})_{\eta \in \mathcal{T}_{\alpha}, k > i} \equiv_{A} (a_{\eta,j})_{\eta \in \mathcal{T}_{\alpha}, j \leq i} (a_{\eta,k})_{\eta \in \mathcal{T}_{\alpha}, k > i}$.
    \item The following independence holds: 
    $$
    (b'_{\eta})_{\eta \in \mathcal{T}_{\alpha}} \ind_{A} (a'_{\eta,j})_{\eta \in \mathcal{T}_{\alpha}, j \leq i} (a_{\eta,k})_{\eta \in \mathcal{T}_{\alpha}, k > i}
    $$
\end{enumerate}
Fix $i < \gamma$ and suppose we have chosen $(a'_{\eta,j})_{\eta \in \mathcal{T}_{\alpha}}$ for all $j < i$. Pick $(a'_{\eta,i})_{\eta \in \mathcal{T}_{\alpha}}$ $s$-indiscernible over $A \cup \{a'_{\eta,j} : \eta \in \mathcal{T}_{\alpha}, j < i\} \cup \{a_{\eta,k} : \eta \in \mathcal{T}_{\alpha}, k > i\} \cup \{b'_{\eta} : \eta \in \mathcal{T}_{\alpha}\}$ and locally based on $(a_{\eta,i})_{\eta \in \mathcal{T}_{\alpha}}$. Then (1) is satisfied and (2) is easy to check using local basedness and the inductive assumption.  We assumed $I_{\gamma}$ was mutually $s$-indiscernible over $A$ and hence by (3) of the inductive hypothesis, we know that $(a_{\eta,i})_{\eta \in \mathcal{T}_{\alpha}}$ is $s$-indiscernible over 
$$
A \cup \{a'_{\eta,j} : \eta \in \mathcal{T}_{\alpha}, j < i\} \cup \{a_{\eta,k} : \eta \in \mathcal{T}_{\alpha}, k > i\}
$$
and therefore $(a'_{\eta,i})_{\eta \in \mathcal{T}_{\alpha}}$ has the same type over this set, which establishes (3).  Finally, (4) follows by local basedness, (3), and the invariance of non-forking independence. More explicitly, suppose there is a finite tuple $b$ from $(b'_{\eta})_{\eta \in \mathcal{T}_{\alpha}}$, a finite tuple $a$ from $\{a'_{\eta,j} : j < i\} \cup\{a_{\eta,k} : k > i\}$, and a finite tuple $\overline{\eta}$ from $\mathcal{T}_{\alpha}$ such that 
$$
\models \varphi(b;a'_{\overline{\eta},i},a)
$$
where $\varphi(x;y,z) \in L(A)$ is a formula such that $\varphi(x;a'_{\overline{\eta},i},a)$ forks over $A$. Local basedness entails that there is $\overline{\nu}$ with $\mathrm{qftp}_{L_{s,\alpha}}(\overline{\nu}) = \mathrm{qftp}_{L_{s,\alpha}}(\overline{\eta})$ such that 
$$
\models \varphi(b;a_{\overline{\nu},i},a).
$$
But by mutual $s$-indiscernibility, $a_{\overline{\nu},i} \equiv_{Aa} a_{\overline{\eta},i}$ and, by (3), $a_{\overline{\eta},i} \equiv_{Aa} a'_{\overline{\eta},i}$ and hence $\varphi(x;a_{\overline{\nu},i},a)$ forks over $A$ as well. This contradicts the inductive hypothesis that 
$$
    (b'_{\eta})_{\eta \in \mathcal{T}_{\alpha}} \ind_{A} (a'_{\eta,j})_{\eta \in \mathcal{T}_{\alpha}, j < i} (a_{\eta,k})_{\eta \in \mathcal{T}_{\alpha}, k \geq i}.
    $$
 This shows that our choice of $(a'_{\eta,i})_{\eta \in \mathcal{T}_{\alpha}}$ satisfies the requirements.  

 Having constructed our sequence $I'_{\gamma} = \langle (a'_{\eta,i})_{\eta \in \mathcal{T}_{\alpha}} : i < \gamma\rangle$, we have $I'_{\gamma} \equiv_{A} I_{\gamma}$ by (3) and $(b'_{\eta})_{\eta \in \mathcal{T}_{\alpha}}$ is $s$-indiscernible over $I'_{\gamma}$ by (2). Moreover, each $(a_{\eta,i})_{\eta \in \mathcal{T}_{\alpha}}$ is indiscernible over $A(b'_{\eta})_{\eta \in \mathcal{T}_{\alpha}}(a_{\eta,j})_{\eta \in \mathcal{T}_{\alpha},j \neq i}$ by (1). Finally, by (4), we have $(b'_{\eta})_{\eta \in \mathcal{T}_{\alpha}} \ind_{A} I'_{\gamma}$.  Choosing $(a_{\eta,\gamma})$ such that 
 $$
 I_{\gamma} (a_{\eta,\gamma})_{\eta \in \mathcal{T}_{\alpha}} \equiv_{A} I'_{\gamma} (b'_{\eta})_{\eta \in \mathcal{T}_{\alpha}}
 $$
 we arrive at $I_{\gamma+1}$.  There is nothing to do at limits, so we have succeeded in constructing our sequence $I_{\kappa}$.  Applying Erd\H{o}s-Rado to $I_{\kappa}$, then, we obtain the desired sequence $I$. 
\end{proof}

\begin{lem} \label{mutual preservation}
    Suppose $(a_{\eta})_{\eta \in \mathcal{T}_{\alpha+1}}$ is a tree of tuples such that $I = \langle a_{\unrhd \langle i \rangle} : i < \omega \rangle$ is mutually $s$-indiscernible and Morley over $A$. Then if $(a'_{\eta})_{\eta \in \mathcal{T}_{\alpha+1}}$ is $s$-indiscernible and locally based on $(a_{\eta})_{\eta \in \mathcal{T}_{\alpha+1}}$ over $A$ and $I' = \langle a'_{\unrhd \langle i \rangle}: i < \omega\rangle$, then $I' \equiv_{A} I$ and thus $I'$ is also mutually $s$-indiscernible and Morley over $A$. 
\end{lem}

\begin{proof}
Suppose $\overline{\eta}$ and $\overline{\nu}$ are tuples from $\mathcal{T}_{\alpha+1} \setminus \{\emptyset\}$ with $\mathrm{qftp}_{L_{s,\alpha+1}}(\overline{\eta}) = \mathrm{qftp}_{L_{s,\alpha+1}}(\overline{\nu})$.  After possibly reordering the tuples, there are $i_{0} < \ldots < i_{k-1}$ and $j_{0} < \ldots < j_{k-1}$ such that $\overline{\eta} = (\overline{\eta}_{0}, \ldots, \overline{\eta}_{k-1})$ and $\overline{\nu} = (\overline{\nu}_{0}, \ldots, \overline{\nu}_{k-1})$ where each $\overline{\eta}_{l}$ comes from the tree $\unrhd \langle i_{l} \rangle$ and $\overline{\nu}_{l}$ comes from the tree $\unrhd \langle j_{l} \rangle$ for $l < k$. Then, in particular, $\mathrm{qftp}_{L_{s,\alpha+1}}(\overline{\eta}_{l}) = \mathrm{qftp}_{L_{s,\alpha+1}}(\overline{\nu}_{l})$ for all $l < k$. Additionally, for all $l < k$, let $\overline{\eta}'_{l}$ be the element of the tree $\unrhd \langle j_{l} \rangle$ corresponding to $\overline{\eta}_{l}$ (i.e. replace each node $\langle i_{l}\rangle^{\frown} \xi$ enumerated in $\overline{\eta}_{l}$ with $\langle j_{l} \rangle^{\frown} \xi$). Because $I$ is an $A$-indiscernible sequence, we have 
$$
(a_{\overline{\eta}_{0}}, \ldots, a_{\overline{\eta}_{k-1}}) \equiv_{A} (a_{\overline{\eta}'_{0}}, \ldots, a_{\overline{\eta}_{k-1}'}). 
$$
Additionally, in the tree $\unrhd \langle j_{l} \rangle$ (naturally viewed as an $L_{s,\alpha}$-structure), we have $\mathrm{qftp}_{L_{s,\alpha}}(\overline{\eta}_{l}') = \mathrm{qftp}_{L_{s,\alpha}}(\overline{\eta}_{l})$ for all $l < k$.  Thus, mutual $s$-indiscernibility entails 
$$
(a_{\overline{\eta}'_{0}}, \ldots, a_{\overline{\eta}_{k-1}'}) \equiv_{A} (a_{\overline{\nu}_{0}}, \ldots, a_{\overline{\nu}_{k-1}}).
$$
Thus we have shown that $a_{\overline{\eta}} \equiv_{A} a_{\overline{\nu}}$.  Therefore, it follows, by local basedness, that $I \equiv_{A} I'$ and the result follows. 
\end{proof}

\begin{defn}\label{weaklyspread}
Suppose $(a_{\eta})_{\eta \in \mathcal{T}_{\alpha}}$ is a tree of tuples in $\mathbb{M}$, and $A$ is a set of parameters.  
\begin{enumerate}
\item We say $(a_{\eta})_{\eta \in \mathcal{T}_{\alpha}}$ is \emph{weakly spread out over} $A$ 
if for all $\eta \in \mathcal{T}_{\alpha}$ with $\text{dom}(\eta) =[\beta+1,\alpha)$ for
some $\beta \in [\alpha]$, the sequence of cones
$(a_{\unrhd \eta \frown \langle i \rangle})_{i < \omega}$ is a Morley sequence 
in $\mathrm{tp}(a_{\unrhd \eta \frown \langle 0 \rangle}/A)$.
\item Suppose $(a_{\eta})_{\eta \in \mathcal{T}_{\alpha}}$ is a tree which is weakly 
spread out and $s$-indiscernible over $A$ and for all pairs of finite subsets $w,v$ of $\alpha$ with $|w| = |v|$,
$$
(a_{\eta})_{\eta \in \mathcal{T}_{\alpha} \upharpoonright w} \equiv_{A} (a_{\eta})_{\eta \in \mathcal{T}_{\alpha} \upharpoonright v}
$$
then we say $(a_{\eta})_{\eta \in \mathcal{T}_{\alpha}}$ is a \emph{weakly Morley tree} over $A$.  
\item A \emph{weak tree Morley sequence} over $A$ is a $A$-indiscernible sequence of the form $(a_{\zeta_\beta})_{\beta \in [\alpha]}$ for some 
weakly Morley tree $(a_{\eta})_{\eta \in \mathcal{T}_{\alpha}}$ over $A$. More generally, we will say an $A$-indiscernible sequence $I$ is a weak tree Morley sequence over $A$ if it is EM-equivalent to a sequence of this form.  
\end{enumerate}
\end{defn}

\begin{rem}
If $I = \langle b_{i} : i < \omega \rangle$ is an $A$-indiscernible sequence and $I \equiv_{A} J$ for some weak tree Morley sequence $J$ over $A$, then $I$ is a weak tree Morley sequence over $A$.  In particular, if $I$ is a subsequence of $J$, by the $A$-indiscernibility of $J$, the sequence $I$ is also weak tree Morley over $A$.  
\end{rem}

\begin{fact} \label{slightly less basic Kim-independence facts}
Suppose $T$ is NSOP$_{1}$ with existence and $A$ is a set of parameters.
\begin{enumerate}
\item If $a \ind^{K}_{A} b$, there is an $Ab$-indiscernible sequence $I = \langle a_{i} : i < \omega \rangle$ over $A$ with $a_{0} = a$ such that $I$ is weak tree Morley over $A$. \cite[Lemma 4.7]{dobrowolski2019independence}
\item Kim's Lemma for weak tree Morley sequences:  the fomula $\varphi(x;a_{0})$ Kim-divides over $A$ if and only if $\{\varphi(x;a_{i}) : i < \omega\}$ is inconsistent for some weak tree Morley sequence $\langle a_{i} : i < \omega \rangle$ over $A$ if and only if $\{\varphi(x;a_{i}) : i < \omega\}$ is inconsistent for all weak tree Morley sequences $\langle a_{i} : i < \omega \rangle$ over $A$.  \cite[Corollary 4.8]{dobrowolski2019independence}
\item If $a \equiv^{L}_{A} b$ and $a \ind^{K}_{A} b$, there is an $\ind^{K}$-Morley sequence over $A$ starting with $(a,b)$ as its first two elements (follows from Fact \ref{basic Kim-independence facts} as in \cite[Corollary 6.6]{kaplan2017kim}).
\end{enumerate}
\end{fact}

\subsection{Further properties of Kim-independence}

%

\begin{fact} \cite[Lemma 5.7]{dobrowolski2019independence} \label{strong extension} 
Suppose $T$ is NSOP$_{1}$ with existence.  If $A$ is a set of parameters, $c$ is an arbitrary tuple, and $a \ind^{K}_{A} b$, then there is $a' \equiv^{L}_{Ab} a$ such that $a' \ind^{K}_{A} bc$. 
\end{fact}

The following lemma is easy and well-known, but, in the absence of a clear reference, we provide a proof:  

\begin{lem} \label{type definability of kim-independence}
Suppose $T$ is NSOP$_{1}$ with existence, $A$ is a set of parameters, and $p(x) \in S(A)$.  
\begin{enumerate}
\item Given any tuple of variables $y$, there is a partial type $\Gamma(x,y)$ over $A$ such that $(a,b) \models \Gamma(x,y)$ if and only if $a \models p$ and $a \ind^{K}_{A} b$. 
\item There is a partial type $\Delta(x_{i} : i < \omega)$ over $A$ such that $I = \langle a_{i} : i < \omega \rangle \models \Delta$ if and only if $I$ is an $\ind^{K}$-Morley sequence over $A$ in $p$.  
\end{enumerate}
\end{lem}

\begin{proof}
(1)  By compactness, we may assume $y$ is finite.  Fix $c \models p$ and define $\Gamma(x,y)$ by 
$$
\Gamma(x,y) = p(x) \cup \{\neg \varphi(y;x) : \varphi(y;c) \text{ Kim-divides over }A\}.
$$
By symmetry, invariance, and Kim-forking = Kim-dividing, this partial type is as desired. 

(2) One can take $\Delta$ to be the partial type that asserts $\langle x_{i} : i < \omega \rangle$ is $A$-indiscernible, every $x_{i} \models p$, and $x_{i} \ind^{K}_{A} x_{<i}$ (which is type-definable over $A$ by (1)).  
\end{proof}

The following lemma is the analogue of the `strong independence theorem' of \cite[Theorem 2.3]{kruckman2018generic} for Lascar strong types.  

\begin{lem} \label{strongIT}
Suppose $T$ is NSOP$_{1}$ with existence.  If $A$ is a set of parameters, $a_{0} \ind^{K}_{A} b$, $a_{1} \ind^{K}_{A} c$, $b \ind^{K}_{A} c$, and $a_{0} \equiv^{L}_{A} a_{1}$, then there is $a$ such that $a \equiv^{L}_{Ab} a_{0}$, $a \equiv^{L}_{Ac} a_{1}$ and, additionally, we have $a \ind^{K}_{A} bc$, $b \ind^{K}_{A} ac$, and $c \ind^{K}_{A} ab$.  
\end{lem}

\begin{proof}
By Fact \ref{strong extension}, there is $c' \equiv^{L}_{Ab} c$ such that $c' \ind^{K}_{A} bc$.  Let $\sigma \in \mathrm{Autf}(\mathbb{M}/Ab)$ be an automorphism such that $\sigma(c') = c$ and let $c_{0} = \sigma(c)$.  Then we have $c \ind^{K}_{A} bc_{0}$ and $c_{0} \equiv_{Ab}^{L} c$ and hence, in particular, $c_{0}b \equiv_{A}^{L} cb$.  By symmetry and a second application of Fact \ref{strong extension} once again, we find $b''c'' \equiv^{L}_{Ac} bc_{0}$ with $b''c'' \ind^{K}_{A} bc$.  Let $\tau \in \mathrm{Autf}(\mathbb{M}/Ac)$ be a strong automorphism with $\tau(b''c'') = bc_{0}$ and define $b_{1} = \tau(b)$.  Then by construction, we have $b''c'' \equiv_{A} bc_{0}$ and $bc_{0} \equiv^{L}_{A} bc$, it follows that $b''c'' \equiv^{L}_{A} bc$, and hence $\tau(b''c'') \equiv^{L}_{A} \tau(bc)$, which, after unraveling definitions, gives $bc_{0} \equiv^{L}_{A} b_{1}c$.  Moreover, since $b''c'' \ind^{K}_{A} bc$, we obtain $bc_{0} \ind^{K}_{A} b_{1}c$ by invariance.   Let $b_{0} = b$ and $c_{1} = c$.  By Fact \ref{slightly less basic Kim-independence facts}(3), we can extend the sequence $\langle (b_{i},c_{i}) : i < 2 \rangle$ to a weak tree Morley sequence $I = \langle (b_{i},c_{i}) : i \in \mathbb{Z} \rangle$ over $A$.  

Choose $a'$ such that $a_{1}c_{1} \equiv^{L}_{A} a'c_{0}$.  Then we have $a_{0} \equiv^{L}_{A} a'$, as well as $a_{0} \ind^{K}_{A} b_{0}$, $a' \ind^{K}_{A} c_{0}$ by our assumptions.  Additionally, since $b \ind^{K}_{A} c$, $b_{0} = b$ and $c \equiv_{Ab} c_{0}$, we have $b_{0} \ind^{K}_{A} c_{0}$.  Therefore, by Fact \ref{basic Kim-independence facts}(6), there is $a_{*}$ with $a_{*} \equiv^{L}_{Ab_{0}} a_{0}$, $a_{*} \equiv^{L}_{Ac_{0}} a'$, with $a_{*} \ind^{K}_{A} b_{0}c_{0}$.    

Because $I$ is a weak tree Morley sequence and $a_{*} \ind^{K}_{A} b_{0}c_{0}$, by Kim's lemma, compactness, and an automorphism, there is $a_{**} \ind^{K}_{A} I$ such that $a_{**}b_{0}c_{0} \equiv^{L}_{A} a_{*}b_{0}c_{0}$ and such that $I$ is $Aa_{**}$-indiscernible.  Note that, by construction, $a_{**} \equiv^{L}_{Ab} a_{0}$, $a_{**} \equiv^{L}_{Ac} a_{1}$, and $a_{**} \ind^{K}_{A} bc$.

Additionally, the sequence $\langle b_{i} : i \in \mathbb{Z}^{\leq 0} \rangle$ is a weak tree Morley sequence over $A$ which is $Aa_{**}c$-indiscernible and containing $b_{0} = b$, hence $a_{**} c \ind^{K} b$, by Kim's lemma for weak tree Morley sequences.  Similarly, the sequence $\langle c_{i} : i \in \mathbb{Z}^{\geq 1} \rangle$ is a weak tree Morley sequence over $A$ containing $c = c_{1}$ which is $Aa_{**}b$-indiscernible, yielding $a_{**}b \ind^{K}_{A} c$.  By symmetry, we conclude. 
\end{proof}

\section{Transitivity and witnessing} \label{transitivity section}

\subsection{Preliminary lemmas}

We begin by establishing some lemmas, allowing us to construct sequences that are $\ind^{K}$-Morley over more than one base simultaneously.  The broad structure of the argument will follow that of \cite{kaplan2019transitivity}, which established transitivity over models for Kim-independence in NSOP$_{1}$ theories, however all uses of coheirs and heirs will need to be replaced.  

In particular, the following lemma does not follow the corresponding \cite[Lemma 3.1]{kaplan2019transitivity}, instead producing the desired sequence by a tree-induction.  

\begin{lem} \label{good sequence}
Suppose $T$ is NSOP$_{1}$ and satisfies the existence axiom.  If $A \subseteq B$ and $a \ind^{K}_{A} B$, then there is a weak tree Morley sequence $\langle a_{i} : i < \omega \rangle$ over $B$ with $a_{0} = a$ such that $a_{i} \ind^{K}_{A} Ba_{<i}$ for all $i < \omega$.  	
\end{lem}

\begin{proof}
By induction on $\alpha$, we will construct trees $(a^{\alpha}_{\eta})_{\eta \in \mathcal{T}_{\alpha}}$ so that 
\begin{enumerate}
\item $(a^{\alpha}_{\eta})_{\eta \in \mathcal{T}_{\alpha}}$ is $s$-indiscernible and weakly spread out over $B$.
\item $a^{\alpha}_{\eta} \models \text{tp}(a/B)$ for all $\eta \in \mathcal{T}_{\alpha}$.  
\item If $\alpha$ is a successor, $a^{\alpha}_{\emptyset} \ind^{K}_{A} Ba^{\alpha}_{\vartriangleright \emptyset}$.
\item If $\alpha < \beta$, then $a^{\beta}_{\iota_{\alpha\beta}(\eta)} = a^{\alpha}_{\eta}$ for all $\eta \in \mathcal{T}_{\alpha}$. 	
\end{enumerate}
For $\alpha = 0$, we put $a^{0}_{\emptyset} = a$, and for $\delta$ limit, we will define $(a^{\delta}_{\eta})_{\eta \in \mathcal{T}_{\delta}}$ by setting $a^{\delta}_{\iota_{\alpha \delta}(\eta)} = a^{\alpha}_{\eta}$ for all $\alpha < \delta$ and $\eta \in \mathcal{T}_{\alpha}$ which, by (4) and induction, is well-defined and satisfies the requirements.  

Now suppose we are given $(a^{\beta}_{\eta})_{\eta \in \mathcal{T}_{\beta}}$ satisfying the requirements for all $\beta \leq \alpha$.  Let $\langle (a^{\alpha}_{\eta,i})_{\eta \in \mathcal{T}_{\alpha}} : i < \omega \rangle$ be a mutually $s$-indiscernible Morley sequence over $B$ with $a^{\alpha}_{\eta,0} = a^{\alpha}_{\eta}$ for all $\eta \in \mathcal{T}_{\alpha}$, which exists by Lemma \ref{Morley mutual}.  Apply extension to find $a_{*} \equiv_{B} a$ so that $a_{*} \ind^{K}_{A} B(a^{\alpha}_{\eta,i})_{\eta \in \mathcal{T}_{\alpha},i < \omega}$.  Define a tree $(b_{\eta})_{\eta \in \mathcal{T}_{\alpha+1}}$ by setting $b_{\emptyset}= a_{*}$ and $b_{\langle i \rangle \frown \eta} = a^{\alpha}_{\eta,i}$ for all $i < \omega$ and $\eta \in \mathcal{T}_{\alpha}$.  We may define $(a^{\alpha+1}_{\eta})_{\eta \in \mathcal{T}_{\alpha+1}}$ to be a tree which is $s$-indiscernible over $B$ and locally based on $(b_{\eta})_{\eta \in \mathcal{T}_{\alpha+1}}$ over $B$.  By an automorphism, we may assume $a^{\alpha+1}_{\iota_{\alpha \alpha+1}(\eta)} = a^{\alpha}_{\eta}$ for all $\eta \in \mathcal{T}_{\alpha}$, hence conditions (2), and (4) are clearly satisfied.  Moreover, by Lemma \ref{mutual preservation}, we have $\langle (a_{\eta,i})_{\eta \in \mathcal{T}_{\alpha}} : i < \omega \rangle \equiv_{B} \langle a^{\alpha+1}_{\unrhd \langle i \rangle} : i < \omega \rangle$ so $\langle a^{\alpha+1}_{\unrhd \langle i \rangle} : i < \omega \rangle$ is a Morley sequence over $B$.  Then by (4) and induction, it follows that $(a^{\alpha+1}_{\eta})_{\eta \in \mathcal{T}_{\alpha+1}}$ is $s$-indiscernible and spread out over $B$, which shows (1).

For (3), we just note that, by symmetry, if $a_{\emptyset}^{\alpha+1} \nind_{A}^{K} Ba^{\alpha+1}_{\vartriangleright \emptyset}$, there is some formula $\varphi(x;a_{\emptyset}^{\alpha+1}) \in \text{tp}(Ba^{\alpha+1}_{\vartriangleright \emptyset}/Aa^{\alpha+1}_{\emptyset})$ that Kim-divides over $A$.  As the tree $(a^{\alpha+1}_{\eta})_{\eta \in \mathcal{T}_{\alpha+1}}$ is locally based on $(b_{\eta})_{\eta \in \mathcal{T}_{\alpha+1}}$, it follows that some tuple from $Bb_{\vartriangleright \emptyset}$ also realizes $\varphi(x;a_{*})$ and $a_{*} \equiv_{B} a^{\alpha+1}_{\emptyset}$, so $\varphi(x;a_{*})$ Kim-divides over $A$ as well, contradicting the choice of $a_{*}$.  This contradiction establishes (3), completing the induction.

By considering $(a^{\kappa}_{\eta})_{\eta \in \mathcal{T}_{\kappa}}$ for $\kappa$ sufficiently large, we may apply Erd\H{o}s-Rado, as in \cite[Lemma 5.10]{kaplan2017kim}, to find the desired sequence.      
\end{proof}

The proof of the next lemma follows \cite[Lemma 3.2]{kaplan2019transitivity}.

\begin{lem} \label{technical lemma}
Suppose $T$ is an NSOP$_{1}$ theory satisfying the existence axiom.  If $a \ind^{K}_{A} b$ and $c \ind^{K}_{A} b$, then there is $c'$ so that $c' \equiv_{Ab} c$, $ac' \ind^{K}_{A} b$, and $a \ind^{K}_{Ab} c'$.  	
\end{lem}

\begin{proof}
Define a partial type $\Gamma(x;b,a)$ over $Aab$ as follows:
$$
\Gamma(x;b,a) = \text{tp}(c/Ab) \cup \{\neg \varphi(x,a;b) : \varphi(x,y;b) \in L(Ab) \text{ Kim-divides over }A\}
$$

\textbf{Claim 1}:  If $\langle a_{i} : i < \omega \rangle$ is an $Ab$-indiscernible sequence satisfying $a_{0} = a$ and $a_{i} \ind^{K}_{A} ba_{<i}$ for all $i < \omega$, then $\bigcup_{i < \omega} \Gamma(x;b,a_{i})$ is consistent.

\emph{Proof of claim}:  By induction on $n < \omega$, we will find $c_{n} \equiv^{L}_{A} c$ such that $c_{n} \ind^{K}_{A} ba_{<n}$ and $c_{n} \models \bigcup_{i < n} \Gamma(x;b,a_{i})$.  For $n = 0$, we can put $c_{0} = c$, since $c \ind^{K}_{A} b$ by assumption.  Assume we have found $c_{n}$, and, by Fact \ref{strong extension}, choose $c'$ such that $c' \equiv^{L}_{A} c$ and $c' \ind^{K}_{A} a_{n}$.  Then $c' \equiv^{L}_{A} c \equiv^{L}_{A} c_{n}$ and, since $a_{n} \ind^{K}_{A} ba_{<n}$, we may apply Lemma \ref{strongIT} to find $c_{n+1} \equiv^{L}_{A} c$ such that $c_{n+1} \models \text{tp}(c_{n}/Aba_{<n}) \cup \text{tp}(c'/Aa_{n})$ and such that $c_{n+1} \ind^{K}_{A} ba_{<n+1}$ and $a_{n}c_{n+1} \ind^{K}_{A} ba_{<n}$, hence, in particular, $c_{n+1} \models \bigcup_{i < n+1} \Gamma(x;b,a_{i})$.  The claim follows by compactness.\qed

Next we define a partial type $\Delta(x;b,a)$ as follows:
$$
\Delta(x;b,a) = \Gamma(x;b,a) \cup \{\neg\psi(x;b,a) : \psi(x;b,a) \in L(Aab) \text{ Kim-divides over }Ab\}.
$$  
\textbf{Claim 2}:  The set of formulas $\Delta(x;b,a)$ is consistent.

\emph{Proof of claim}:  Suppose not.  Then because Kim-forking and Kim-dividing are the same in NSOP$_{1}$ with existence, there is some formula $\psi(x;b,a) \in L(Aab)$ such that 
$$
\Gamma(x;b,a) \vdash \psi(x;b,a)
$$
and $\psi(x;b,a)$ Kim-divides over $Ab$.  As $a \ind^{K}_{A} b$, we know by Lemma \ref{good sequence} that there is a sequence $\langle a_{i} : i < \omega \rangle$ with $a_{0} = a$ which is a weak tree Morley sequence over $Ab$ and satisfies $a_{i} \ind^{K}_{A} ba_{<i}$ for all $i < \omega$.  Then by Claim 1, $\bigcup_{i < \omega} \Gamma(x;b,a_{i})$ is consistent.  However, we have
$$
\bigcup_{i < \omega} \Gamma(x;b,a_{i}) \vdash \{\psi(x;b,a_{i}) : i < \omega\}
$$
and $\{\psi(x;b,a_{i}) : i < \omega\}$ is inconsistent, because weak tree Morley sequences witness Kim-dividing.  This contradiction proves the claim.\qed

To conclude, we may take $c'$ to be any realization of $\Delta(x;b,a)$.  
\end{proof}

The next proposition is a strengthening of Fact \ref{slightly less basic Kim-independence facts}(1).

\begin{prop} \label{extra good sequence}
Suppose $T$ is an NSOP$_{1}$ theory satisfying the existence axiom.  If $a \ind^{K}_{A} b$, then there is a sequence $I = \langle a_{i} : i < \omega \rangle$ with $a_{0} = a$ such that $I$ is a weak tree Morley sequence over $A$ and an $\ind^{K}$-Morley sequence over $Ab$.\end{prop}

\begin{proof}
By induction on $\alpha$, we will construct trees $(a^{\alpha}_{\eta})_{\eta \in \mathcal{T}_{\alpha}}$ satisfying the following:
\begin{enumerate}
\item For all $\eta \in \mathcal{T}_{\alpha}$, $a^{\alpha}_{\eta} \models \text{tp}(a/Ab)$.
\item The tree $(a^{\alpha}_{\eta})_{\eta \in \mathcal{T}_{\alpha}}$ is $s$-indiscernible over $Ab$ and weakly spread out over $A$.
\item If $\alpha$ is a successor, then $a^{\alpha}_{\emptyset} \ind^{K}_{Ab} a^{\alpha}_{\vartriangleright \emptyset}$.
\item $(a^{\alpha}_{\eta})_{\eta \in \mathcal{T}_{\alpha}} \ind^{K}_{A} b$.
\item If $\alpha < \beta$, then $a^{\beta}_{\iota_{\alpha \beta}(\eta)} = a^{\alpha}_{\eta}$ for all $\eta \in \mathcal{T}_{\alpha}$.	
\end{enumerate}
Put $a^{0}_{\emptyset} = a$ and for $\delta$ limit, if we are given $(a^{\alpha}_{\eta})_{\eta \in \mathcal{T}_{\alpha}}$ for every $\alpha < \delta$, we can define $(a^{\delta}_{\eta})_{\eta \in \mathcal{T}_{\delta}}$ by setting $a^{\delta}_{\iota_{\alpha\delta}(\eta)} = a^{\alpha}_{\eta}$ for all $\alpha < \delta$ and $\eta \in \mathcal{T}_{\alpha}$, which is well-defined by (5) and is easily seen to satisfy the requirements.

Suppose now we are given $(a^{\alpha}_{\eta})_{\eta \in \mathcal{T}_{\alpha}}$.  Let $J = \langle (a^{\alpha}_{\eta,i})_{\eta \in \mathcal{T}_{\alpha}} : i < \omega \rangle$ be a mutually $s$-indiscernible Morley sequence over $A$ with $(a^{\alpha}_{\eta,0})_{\eta \in \mathcal{T}_{\alpha}} = (a^{\alpha}_{\eta})_{\eta \in \mathcal{T}_{\alpha}}$, which exists by Lemma \ref{Morley mutual}.  By (4), symmetry, and the Chain Condition, Fact \ref{basic Kim-independence facts}(5) we may assume $J$ is $Ab$-indiscernible and $J \ind^{K}_{A} b$.  By Lemma \ref{technical lemma}, there is $a_{*} \equiv_{Ab} a$ such that $a_{*} \ind^{K}_{Ab} J$ and $a_{*}J \ind^{K}_{A} b$.  After defining a tree $(c_{\eta})_{\eta \in \mathcal{T}_{\alpha+1}}$ by $c_{\emptyset} = a_{*}$ and $c_{\langle i \rangle \frown \eta} = a^{\alpha}_{\eta,i}$ for all $\eta \in \mathcal{T}_{\alpha}$, these conditions on $a_{*}$ imply that $(c_{\eta})_{\eta \in \mathcal{T}_{\alpha}}$ satisfy (3) and (4) respectively.  Let $(a^{\alpha+1}_{\eta})_{\eta \in \mathcal{T}_{\alpha+1}}$ be any tree $s$-indiscernible over $Ab$ locally based on $(c_{\eta})_{\eta \in \mathcal{T}_{\alpha}}$ over $Ab$.  This still satisfies (2) by Lemma \ref{mutual preservation}. Moreover, as $c_{\unrhd \langle i \rangle} \models \mathrm{tp}((a^{\alpha}_{\eta})_{\eta \in \mathcal{T}_{\alpha}}/Ab)$ for all $i < \omega$, it follows from local basedness that $a^{\alpha+1}_{\unrhd \langle i \rangle} \models \mathrm{tp}((a^{\alpha}_{\eta})_{\eta \in \mathcal{T}_{\alpha}}/Ab)$ for all $i < \omega$ as well.  Hence, by an automorphism over $Ab$, we may assume $a^{\alpha+1}_{0 \frown \eta} = a^{\alpha}_{\eta}$ for all $\eta \in \mathcal{T}_{\alpha}$, which ensures the constructed tree satisfies (5), and (1)-(4) are easy to verify.

Given $(a^{\kappa}_{\eta})_{\eta \in \mathcal{T}_{\kappa}}$ for $\kappa$ sufficiently large, we may, by Erd\H{o}s-Rado (see, e.g., \cite[Lemma 5.10]{kaplan2017kim}), obtain a weak Morley tree $(b_{\eta})_{\eta \in \mathcal{T}_{\omega}}$ over $A$ satisfying (1)\textemdash(4).  Then the sequence $I = \langle a_{i} : i < \omega \rangle$ defined by $a_{i} = b_{\zeta_{i}}$ for all $i < \omega$ is a weak tree Morley sequence over $A$, as it is a path in a weak Morley tree over $A$, but by (3), we have $a_{i} \ind^{K}_{Ab} a_{<i}$ for all $i$, so $I$ is $\ind^{K}$-Morley over $Ab$ as well. 
\end{proof}

\subsection{Transitivity and witnessing}

The following theorem establishes the transitivity of Kim-independence in NSOP$_{1}$ theories with existence.  

\begin{thm} \label{transitivity theorem}
Suppose $T$ is NSOP$_{1}$ with existence.  Then if $A \subseteq B$, $a \ind^{K}_{A} B$ and $a \ind^{K}_{B} c$, then $a \ind^{K}_{A} Bc$. 
\end{thm}

\begin{proof}
By Proposition \ref{extra good sequence} and the assumption that $a \ind^{K}_{A} B$, there is a sequence $I = \langle a_{i} : i < \omega \rangle$ with $a_{0} = a$ such that $I$ is an $\ind^{K}$-Morley sequence over $B$ and a weak tree Morley sequence over $A$.  As $c \ind^{K}_{B} a$, by symmetry, and $I$ is $\ind^{K}$-Morley over $B$, there is $I' \equiv_{Ba}I$ such that $I'$ is $Bc$-indiscernible.  Because $I'$ is also a weak tree Morley sequence over $A$, it follows by Kim's lemma that $Bc \ind^{K}_{A} a$.  By symmetry, we conclude. 
\end{proof}

\begin{prop} Assume $T$ is NSOP$_1$ with existence. The following are equivalent.
\begin{enumerate}
\item $a \ind^{K}_{A} b$
\item There is a model $M\supseteq A$ such that  $M \ind^{K}_{A} ab$ (or $M \ind_{A} ab$) and $a\ind^{K}_M b$.
\item There is a model $M\supseteq A$ such that  $M\ind^{K}_A a$ (or $M\ind_{A} a$) and $a\ind^{K}_M b$.
\end{enumerate}
\end{prop}
\begin{proof}
(1)$\Rightarrow$(2) Since  $a\ind^{K}_A b$, there is a Morley sequence $I=\langle  a_i :  i<\omega \rangle$ over $A$ with $a_0=a$ such that $I$ is $Ab$-indiscernible.  By \cite[Lemma 2.17]{dobrowolski2019independence}, there is a model $N$ containing $A$ such that $N\ind_A I$ and 
$I$ is a coheir sequence over $N$.  By compactness and extension  we can clearly assume the length of $I$ is arbitrarily large, and $N\ind_{A}Ib$. Hence by the pigeonhole principle,
there is an infinite subsequence $J$ of $I$ such that all the tuples in $J$ have the same type over $Nb$. Thus, for $a'\in J$, we have $a'\ind^{K}_{N} b$ and $N\ind_{A} a'b$.
Hence  $M=f(N)$ is a desired model, where $f$ is   an $Ab$-automorphism sending $a'$
to $a$.

(2)$\Rightarrow$(3) Clear. 

(3)$\Rightarrow$(1) follows from transitivity and symmetry of $\ind^{K}$. 
\end{proof}

\begin{prop} \label{transfer to models}
Suppose $T$ satisfies the existence axiom.  The following are equivalent for a cardinal $\kappa \geq |T|$:
\begin{enumerate}
\item $T$ is NSOP$_{1}$.  
\item There is no increasing continuous sequence $\langle A_{i} : i < \kappa^{+} \rangle$ of parameter sets and finite tuple $d$ such that $|A_{i}| \leq \kappa$ and $d \nind^{K}_{A_{i}} A_{i+1}$ for all $i < \kappa^{+}$.  
\item There is no set $A$ of parameters of size $\kappa^{+}$ and $p(x) \in S(A)$ with $x$ a finite tuple of variables such that for some increasing and continuous sequence of sets $\langle A_{i} : i < \kappa^{+} \rangle$ with union $A$, we have $|A_{i}| \leq \kappa$ and $p$ Kim-divides over $A_{i}$ for all $i < \kappa^{+}$.  
\end{enumerate}
\end{prop}

\begin{proof}
(1)$\implies$(2)  It suffices to show that, given any increasing continuous sequence $\langle A_{i} : i < \kappa^{+} \rangle$ of parameter sets and tuple $d$ such that $|A_{i}| \leq \kappa$ and $d \nind^{K}_{A_{i}} A_{i+1}$ for all $i < \kappa^{+}$, there is a continuous increasing sequence of models $\langle M_{i} : i < \kappa^{+} \rangle$ and a finite tuple $d'$ such that $|M_{i}| \leq \kappa$ and $d' \nind^{K}_{M_{i}} M_{i+1}$ for all $i < \kappa^{+}$.  This follows from Fact \ref{local character}, since the existence of such a sequence of models implies $T$ has SOP$_{1}$.  Moreover, after naming constants, we may assume $\kappa = |T|$.  

So suppose we are given $\langle A_{i} : i < |T|^{+} \rangle$, an increasing continuous sequence of sets of parameters with $|A_{i}| \leq |T|$ for all $i < |T|^{+}$.  Let $A = \bigcup_{i < |T|^{+}} A_{i}$, and suppose further that we are given some tuple $d$ such that $d \nind^{K}_{A_{i}} A_{i+1}$ for all $i < |T|^{+}$.  By induction on $i < |T|^{+}$ we will build increasing and continuous sequences $\langle A'_{i} : i < |T|^{+} \rangle$ and $\langle M_{i} : i < |T|^{+} \rangle$ satisfying the following for all $i < |T|^{+}$:
\begin{enumerate}
\item $A'_{0} = A_{0}$ and $A'_{\leq i} \equiv A_{\leq i}$.  
\item $M_{i}\models T$ with $|M_{i}| = |T|$ and $A'_{i} \subseteq M_{i}$.
\item $A'_{i+1} \ind^{K}_{A'_{i}} M_{i}$.  
\end{enumerate}
To begin, we define $A'_{0} = A_{0}$ and take $M_{0}$ be any model containing $A'_{0}$ of size $|T|$.  Given $A'_{\leq i}$ and $M_{\leq i}$ satisfying the requirements, we pick $A''_{i+1}$ such that $A'_{\leq i} A''_{i+1} \equiv A_{\leq i+1}$.  Then we apply extension, to obtain $A'_{i+1} \equiv_{A'_{\leq i}} A''_{i+1}$ such that $A'_{i+1} \ind^{K}_{A'_{i}} M_{i}$.  Note that $A'_{\leq i+1} \equiv A_{\leq i+1}$.  We define $M_{i+1}$ to be any model containing $A'_{i+1}M_{i}$ of size $|T|$.  This satisfies the requirements.

At limit $\delta$, we define $A'_{\delta} = \bigcup_{i < \delta} A'_{i}$ and $M_{\delta} = \bigcup_{i < \delta} M_{i}$.  This clearly satisfies (1) and (2) and (3) is trivial.  Therefore this completes the construction.  

Let $M = \bigcup_{i < |T|^{+}} M_{i}$.  Choose $d'$ such that $d\langle A_{i} : i < |T|^{+} \rangle \equiv d' \langle A'_{i} : i < |T|^{+} \rangle$, which is possible by (1).  Then we have $d' \nind^{K}_{A'_{i}} A'_{i+1}$ for all $i < |T|^{+}$.  

Towards contradiction, suppose that there is some $i < |T|^{+}$ with the property that $d' \ind^{K}_{M_{i}} M_{i+1}$.  Then, in particular, we have $d' \ind^{K}_{M_{i}} A'_{i+1}$.  Additionally, because $M_{i} \ind^{K}_{A'_{i}} A'_{i+1}$, we know, by symmetry and transitivity, that $A'_{i+1} \ind^{K}_{A'_{i}} d'M_{i}$.  By symmetry once more, we get $d' \ind^{K}_{A'_{i}} A'_{i+1}$, a contradiction.  This shows that $d' \nind^{K}_{M_{i}} M_{i+1}$ for all $i < |T|^{+}$, completing the proof of this direction.

(2)$\implies$(3)  Suppose (3) fails, i.e.~we are given $A$ of size $\kappa^{+}$, $p \in S(A)$, and an increasing continuous sequence of sets $\langle A_{i} : i < \omega \rangle$ such that $|A_{i}| \leq \kappa$ and $p$ Kim-divides over $A_{i}$ for all $i < \kappa^{+}$.  We will define an increasing continuous sequence of ordinals $\langle \alpha_{i} : i < \kappa^{+} \rangle$ such that $\alpha_{i} \in \kappa^{+}$ and $p \upharpoonright A_{\alpha_{i+1}}$ Kim-divides over $A_{\alpha_{i}}$ for all $i < \kappa^{+}$.  We set $\alpha_{0} = 0$ and given $\langle \alpha_{j} : j \leq i \rangle$, we know that there is some formula $\varphi(x;a_{i+1}) \in p$ that Kim-divides over $A_{\alpha_{i}}$, by our assumption on $p$.  Let $\alpha_{i+1}$ be the least ordinal $< \kappa^{+}$ such that $a_{i+1}$ is contained in $A_{\alpha_{i+1}}$.  For limit $i$, if we are given $\langle \alpha_{j} : j < i \rangle$, we put $\alpha_{i} = \sup_{j < i} \alpha_{j}$.  Then we define $\langle A_{i} : i < \kappa^{+} \rangle$ by $A'_{i} = A_{\alpha_{i}}$ for all $i < \kappa$, and let $d \models p$ be any realization.  By construction, we have $d \nind^{K}_{A'_{i}} A'_{i+1}$ for all $i < \kappa^{+}$, which witnesses the failure of (2).  

(3)$\implies$(1) This was established in \cite[Theorem 3.9]{kaplan2019local}.  
\end{proof}

\begin{rem}
In \cite[Proposition 4.6]{casanovasmore} it is shown that in every theory with TP$_{2}$, there is an increasing chain of sets $\langle D_{i} : i < |T|^{+} \rangle$ and tuple $d$ such that $|D_{i}| \leq |T|$ and $d \nind^{K}_{D_{i}} D_{i+1}$ for all $i < |T|^{+}$.  Hence, for non-simple NSOP$_{1}$ theories, the condition of continuity in the statement of Proposition \ref{transfer to models} is essential.  
\end{rem}

The following theorem will be referred to as `witnessing.'  It shows that $\ind^{K}$-Morley sequences are witnesses to Kim-dividing.  Over models this was established in \cite[Theorem 5.1]{kaplan2019transitivity}, however for us it will be deduced as a corollary of Proposition \ref{transfer to models}.  

\begin{thm} \label{witnessing}
Suppose $T$ is NSOP$_{1}$ with existence and $I = \langle a_{i} : i < \omega \rangle$ is an $\ind^{K}$-Morley sequence over $A$.  If $\varphi(x;a_{0})$ Kim-divides over $A$, then $\{\varphi(x;a_{i}) : i < \omega\}$ is inconsistent.
\end{thm}

\begin{proof}
Suppose towards contradiction that $\varphi(x;a_{0})$ Kim-divides over $A$ and $I = \langle a_{i} : i < \omega \rangle$ is an $\ind^{K}$-Morley sequence over $A$ such that $\{\varphi(x;a_{i}) : i < \omega\}$ is consistent.  By naming $A$ as constants, we may assume $|A| \leq |T|$.  We may stretch $I$ such that $I = \langle a_{i} : i < |T|^{+} \rangle$.  Define $A_{i} = Aa_{<i}$.  Then $\langle A_{i} : i < |T|^{+} \rangle$ is increasing and continuous and $|A_{i}| \leq |T|$.  Let $d \models \{\varphi(x;a_{i}) : i < |T|^{+} \}$.  We claim $d \nind^{K}_{A_{i}} A_{i+1}$ for all $i < |T|^{+}$.  If not, then for some $i < |T|^{+}$, we have $d \ind^{K}_{A_{i}} A_{i+1}$, or, in other words $d \ind^{K}_{Aa_{<i}} a_{i}$.  Since $I$ is an $\ind^{K}$-Morley sequence, we also have $a_{i} \ind^{K}_{A} a_{<i}$, hence $da_{<i} \ind^{K}_{A} a_{i}$, by transitivity (Theorem \ref{transitivity theorem}).  This entails, in particular, that $d \ind^{K}_{A} a_{i}$, which is a contradiction, since $\varphi(x;a_{i})$ Kim-divides over $A$.  This completes the proof. 
\end{proof}

\section{Low NSOP$_{1}$ theories} \label{low section}

This section is dedicated to proving that Lascar and Shelah strong types coincide in any low NSOP$_{1}$ theory with existence.  This generalizes the corresponding result of Buechler for low simple theories \cite{buechler1999lascar} (also independently discovered by Shami \cite{shami2000definability}).  

\begin{defn}
We say that the theory $T$ is \emph{low} if, for every formula $\varphi(x;y)$, there is some $k < \omega$, such that if $I = \langle a_{i} : i < \omega \rangle$ is an indiscernible sequence and $\{\varphi(x;a_{i}) : i < \omega\}$ is inconsistent, then it is $k$-inconsistent.
\end{defn}

In \cite{buechler1999lascar}, the definition of lowness is given in terms of the finiteness of certain $D(p,\varphi)$ ranks, which we will not need here.  However, as observed in \cite{buechler1999lascar}, the above definition coincides with this definition in the case that $T$ is simple.  

\begin{lem} \label{weird witnessing}
Suppose $T$ is NSOP$_{1}$ with existence.  Assume we are given tuples $(a_{i})_{i \leq n}$ and $L(A)$-formulas $(\varphi_{i}(x;y_{i}))_{i \leq n}$ such that $a_{i} \ind^{K}_{A} a_{<i}$ for all $i \leq n$.  Then the following are equivalent:
\begin{enumerate}
\item The formula $\bigwedge_{i \leq n} \varphi_{i}(x;a_{i})$ does not Kim-divide over $A$.
\item For all $A$-indiscernible sequences $\langle \overline{a}_{j} : j < \omega \rangle = \langle (a_{j,0},\ldots, a_{j,n}) : j < \omega \rangle$ with $(a_{0,0},\ldots, a_{0,n}) = (a_{0},\ldots, a_{n})$ and $a_{j,i} \ind^{K}_{A} \overline{a}_{<j} a_{j,<i}$ for all $j < \omega$ and $i \leq n$, the following set of formulas does not Kim-divide over $A$:
$$
\left\{ \bigwedge_{i \leq n} \varphi_{i}(x;a_{j,i}) : j < \omega \right\}.
$$
\item There is an $A$-indiscernible sequence $\langle \overline{a}_{j} : j < \omega \rangle = \langle (a_{j,0},\ldots, a_{j,n}) : j < \omega \rangle$ with $(a_{0,0},\ldots, a_{0,n}) = (a_{0},\ldots, a_{n})$ and $a_{j,i} \ind^{K}_{A} \overline{a}_{<j} a_{j,<i}$ for all $j < \omega$, $i \leq n$ such that 
$$
\left\{ \bigwedge_{i \leq n} \varphi_{i}(x;a_{j,i}) : j < \omega \right\}
$$
is consistent.
\end{enumerate}
\end{lem}

\begin{proof}
(1)$\implies$(2)  Suppose we are given $\langle \overline{a}_{j} : j < \omega \rangle$ as in (2) and let $c \models \bigwedge_{i \leq n} \varphi_{i}(x;a_{i})$ with $c \ind^{K}_{A} \overline{a}_{0}$.  As $\langle \overline{a}_{j} : j < \omega \rangle$ is $A$-indiscernible, for each $j > 0$, there is some $\sigma_{j} \in \mathrm{Autf}(\mathbb{M}/A)$ with $\sigma_{j}(\overline{a}_{0}) = \overline{a}_{j}$.  Define $c_{0} = c$ and $c_{j} = \sigma_{j}(c)$ for all $j > 0$.  Then we have $c_{j} \equiv^{L}_{A} c_{0}$ and $c_{j} \models \bigwedge_{i \leq n} \varphi_{i}(x;a_{j,i})$ for all $j < \omega$.  By inductively applying the independence theorem (with respect to the lexicographic order on $\omega \times n$), we obtain $c_{*} \models \{\varphi_{i}(x;a_{j,i}) : i \leq n, j < \omega\}$ with $c_{*} \ind^{K}_{A} \overline{a}_{<\omega}$, which establishes (2).

(2)$\implies$(3)  It suffices to show that there is an $A$-indiscernible sequence $\langle \overline{a}_{j} : j < \omega \rangle = \langle (a_{j,0},\ldots, a_{j,n}) : j < \omega \rangle$ with $(a_{0,0},\ldots, a_{0,n}) = (a_{0},\ldots, a_{n})$ and such that $a_{j,i} \ind^{K}_{A} \overline{a}_{<j} a_{j,<i}$ for all $j < \omega$, $i \leq n$.  

First, we construct by induction a sequence $\langle (a'_{j,0},\ldots, a'_{j,n} : j < \omega \rangle = \langle \overline{a}'_{j} : j < \omega \rangle$ with $\overline{a}'_{j} \equiv_{A} (a_{0},\ldots, a_{n})$ and $a'_{j,i} \ind^{K}_{A}  \overline{a}'_{<j}a_{j,<i}$ for all $j < \omega$, $i \leq n$.  Given $\overline{a}'_{\leq k}$, we apply extension to find $a'_{k+1,0} \equiv_{A} a_{0}$ with $a'_{k+1,0} \ind^{K}_{A} \overline{a}'_{\leq k}$.  Given $\overline{a}'_{k+1,\leq i}$ for $i < n$, we find $b_{k+1,i+1}$ such that $a'_{k+1,\leq i}b_{k+1,i+1}  \equiv_{A} a_{\leq i}a_{i+1}$.  By invariance, this implies $b_{k+1,i+1} \ind^{K}_{A} a'_{k+1,\leq i}$.  Applying extension once more, we can find $a'_{k+1,i+1} \equiv_{Aa'_{k+1,\leq i}} b_{k+1,i+1}$ such that $a'_{k+1,i+1} \ind^{K}_{A} \overline{a}'_{\leq k} a'_{k+1,\leq i}$.  This completes the construction of $\langle \overline{a}'_{j} : j < \omega \rangle$.  By Ramsey, compactness, and an automorphism, we can extract an $A$-indiscernible sequence $\langle \overline{a}_{j} : j < \omega \rangle$ with $\overline{a}_{0} = (a_{0},\ldots, a_{n})$ as desired.  

(3)$\implies$(1)  Let $c \models \left\{ \bigwedge_{i \leq n} \varphi_{i}(x;a_{j,i}) : j < \omega \right\}$.  Note that, for each $i \leq n$, the sequence $\langle a_{j,i} : j < \omega \rangle$ is an $\ind^{K}$-Morley sequence over $A$.  Hence, by witnessing, Theorem \ref{witnessing}, and the fact that $c \models \{\varphi_{i}(x;a_{j,i}) : j < \omega \}$, we see that $c \ind^{K}_{A} a_{i}$ for each $i \leq n$.  

By induction on $k \leq n$, we will choose $c_{k} \equiv^{L}_{A} c$ such that $c_{k} \models \{\varphi_{i}(x;a_{i}) : i \leq k\}$ and $c_{k} \ind^{K}_{A} a_{\leq k}$.  To begin we set $c_{0} = c$.  Given $c_{k}$ for some $k < n$, we apply the independence theorem to find $c_{k+1}$ with $c_{k+1} \equiv_{Aa_{\leq k}}^{L} c_{k}$,  $c_{k+1} \equiv^{L}_{Aa_{k+1}} c$, and $c_{k+1} \ind^{K}_{A} a_{\leq k+1}$.  After $n$ steps, we obtain $c_{n} \models \bigwedge_{i \leq n} \varphi_{i}(x;a_{i})$ with $c_{n} \ind^{K}_{A} a_{\leq n}$, which shows (1).  
\end{proof}

\begin{rem}
The independence conditions of (2) and (3) do not imply that the sequence $\langle (a_{i,0},\ldots, a_{i,n}) : i < \omega \rangle$ is $\ind^{K}$-Morley, due to the lack of base monotonicity.  Consequently, this lemma strengthens witnessing, Theorem \ref{witnessing}, for formulas of a certain form, showing that they Kim-divide along sequences that are themselves not necessarily $\ind^{K}$-Morley sequences.  
\end{rem}

\begin{cor} \label{type definable}
Suppose $T$ is NSOP$_{1}$ with existence.  Assume that for each $i \leq n$, we are given a complete type $p(y_{i}) \in S(A)$ and an $L(A)$-formula $\varphi_{i}(x;y_{i})$.
\begin{enumerate}
\item There is a partial type $R(y_{0},\ldots, y_{n})$ over $A$ containing $\bigcup_{i \leq n} p_{i}(y_{i})$ such that $a'_{0},\ldots, a'_{n} \models R(y_{0},\ldots, y_{n})$ if and only if $a'_{i} \ind^{K}_{A} a'_{<i}$ for all $i \leq n$ and $\bigwedge_{i \leq n} \varphi_{i}(x;a'_{i})$ does not Kim-divide over $A$. 
\item If, additionally, $T$ is low, then there is a formula $\gamma(y_{0},\ldots, y_{n})$ over $A$ such that, if $(a'_{0},\ldots, a'_{n}) \models p(y_{0}) \cup \ldots \cup p(y_{n})$ and $a'_{i} \ind^{K}_{A} a'_{<i}$ for all $i\leq n$, then $\mathbb{M} \models \gamma(a'_{0},\ldots, a'_{n})$ if and only if $\bigwedge_{i \leq n} \varphi_{i}(x;a'_{i})$ does not Kim-divide over $A$.  
\end{enumerate}
\end{cor}

\begin{proof}
By Lemma \ref{type definability of kim-independence}(1), there is a partial type $\Lambda(z_{i,j} : i < \omega, j \leq n)$ over $A$ which expresses the following:
\begin{enumerate}[(a)]
\item $z_{i,j} \models p_{j}$ for all $i < \omega$ and $j \leq n$.
\item $z_{i,j} \ind^{K}_{A} \overline{z}_{<i}z_{i,<j}$ for all $i < \omega$ and $j \leq n$, where $\overline{z}_{i} = (z_{i,0},\ldots, z_{i,n})$.
\item The sequence $\langle \overline{z}_{i} : i < \omega \rangle$ is $A$-indiscernible.
\end{enumerate}
Let $\lambda'(y_{0},\ldots, y_{n},\overline{z}_{i} : i < \omega)$ be the partial type given by $\{y_{j}  = z_{0,j} : j \leq n\} \cup \Lambda(\overline{z}_{i} : i < \omega)$. 

To show (1), consider the partial type $R_{0}(y_{0},\ldots, y_{n},\overline{z}_{i} : i < \omega)$ which extends $\lambda'$ and expresses additionally that $\{\bigwedge_{i \leq n} \varphi_{i}(x;z_{i,j}) : i < \omega\}$ is consistent.  Then $R$ may be defined by 
$$
R(y_{0},\ldots, y_{n}) \equiv (\exists \overline{z}_{i} : i < \omega) \bigwedge R_{0}(y_{0},\ldots, y_{n},\overline{z}_{i} ; i < \omega).
$$
This $R$ is clearly type definable and, by Lemma \ref{weird witnessing}(3), $R$ has the desired properties.

For (2), we know, by the lowness of $T$, that there is some $k < \omega$ such that, if, given any $\langle (a'_{i,j})_{j \leq n} : i < \omega \rangle$ is an $A$-indiscernible sequence and 
$$
\left\{ \bigwedge_{j \leq n} \varphi_{i}(x;a_{i,j}) : i < \omega\right\}
$$
is inconsistent, then this set of formulas is $k$-inconsistent.  Consider the partial type $R_{1}(y_{0},\ldots, y_{n},\overline{z}_{i} : i < \omega)$ which extends $\lambda'$ and expresses additionally that $\{\bigwedge_{j \leq n} \varphi_{i}(x;z_{i,j}) : i < \omega\}$ is $k$-inconsistent.  Then we will define $R'$ by 
$$
R'(y_{0},\ldots, y_{n}) \equiv (\exists \overline{z}_{i} : i < \omega) \bigwedge R_{1}(y_{0},\ldots, y_{n},\overline{z}_{i} ; i < \omega).
$$
It follows that if $(a'_{j})_{j \leq n} \models \bigcup_{j \leq n} p_{j}(y_{j})$ and $a'_{j} \ind^{K}_{A} a'_{<j}$, then $(a'_{j})_{j \leq n} \models R'$ if and only if $\bigwedge_{j \leq n} \varphi_{j}(x;a'_{j})$ Kim-divides over $A$, by Lemma \ref{weird witnessing}(2).  We showed in (1) that the complement of $R'$ is type-definable (by $R$), and therefore, by compactness, we obtain the desired $\gamma$.  
\end{proof}

\begin{thm}
If $T$ is a low NSOP$_{1}$ theory with existence, then Lascar strong types are strong types.  That is, for any (possibly infinite) tuples $a,b$ and small set of parameters $A$, if $a \equiv^{S}_{A} b$ then $a \equiv^{L}_{A} b$.  
\end{thm}

\begin{proof}
Let $A$ be any small set of parameters.  As $T_{A}$ is $G$-compact by Fact \ref{basic Kim-independence facts}(7) and, trivially, $T_{A}$ is low, it suffices to prove the theorem when $a$ and $b$ realize a type $p(x)$ over $\emptyset$, where $x$ is a finite tuple of variables. Let $r(x,y)$ be a partial type, closed under conjunctions, expressing that $x \equiv^{L} y$, i.e.~$r(x,y)$ defines the finest type-definable equivalence relation over $\emptyset$ with boundedly many classes.  Fix $\varphi(x;y) \in r(x,y)$.  Note that, for any $a \models p$, $\varphi(x;a)$ does not divide over $\emptyset$, because if $\langle a_{i} : i < \omega \rangle$ is an indiscernible sequence with $a_{0} = a$, then $a_{0} \equiv^{L} a_{i}$ for all $i< \omega$, hence $a_{0} \models \bigcup_{i < \omega} r(x;a_{i})$.  In particular, $\varphi(x;a)$ does not Kim-divide over $\emptyset$.  

Define a relation $R_{\varphi}(u,v)$ expressing the following:
\begin{enumerate}
\item $u,v \models p$.
\item There exists $v'$ satisfying: 
\begin{enumerate}
\item $v' \equiv^{L} v$.
\item $u \ind^{K} v'$.
\item $\varphi(x;u) \wedge \varphi(x;v')$ does not Kim-divide over $\emptyset$.
\end{enumerate}
\end{enumerate}
Clearly if $v \models p$ and $v \equiv^{L}v'$, then $v' \models p$.  By Lemma \ref{type definability of kim-independence} and Corollary \ref{type definable}(1), there is a partial type $\Gamma(z,w)$ over $\emptyset$ such that $(v',u) \models \Gamma(z,w)$ if and only if $v',u \models p$,  $u \ind^{K} v'$, and $\varphi(x;v') \wedge \varphi(x;u)$ does not Kim-divide over $\emptyset$.  It follows that $R_{\varphi}(u,v)$ if and only if $(\exists v')\left[\bigwedge r(v,v') \wedge \Gamma(v',u)\right]$, which shows $R_{\varphi}(u,v)$ is type-definable.

In a similar fashion, we define a relation $S_{\varphi}(u,v)$ to hold when the following conditions are satisfied:
\begin{enumerate}
\item $u,v \models p$.
\item There is $v'$ satisfying the following: 
\begin{enumerate}
\item $v' \equiv^{L} v$.
\item $u \ind^{K} v'$.
\item $\varphi(x;u) \wedge \varphi(x;v)$ Kim-divides over $\emptyset$.
\end{enumerate}
\end{enumerate}
The type-definability of $S_{\varphi}$ follows from an identical argument, using Corollary \ref{type definable}(2) in the place of Corollary \ref{type definable}(1) (this is where we make use of our hypothesis that $T$ is low).  

From here, our proof follows the argument of \cite{buechler1999lascar}, as presented in \cite[Section 5.2]{kim2013simplicity}.  First, we show the following:

\textbf{Claim 1}:  If $u,v \models p$, then $R_{\varphi}(u,v)$ if and only if $\neg S_{\varphi}(u,v)$.  

\emph{Proof of Claim 1}:  First, it is clear that it is impossible for both $\neg R_{\varphi}(u,v)$ and $\neg S_{\varphi}(u,v)$ to hold since, by extension for Lascar strong types (Fact \ref{strong extension}) there is $v' \equiv^{L} v$ with $u \ind^{K} v'$ and it must be the case that either $\varphi(x,u) \wedge \varphi(x;v')$ Kim-divides or $\varphi(x;u) \wedge \varphi(x;v')$ does not Kim-divide over $\emptyset$.  

Secondly, suppose $R_{\varphi}(u,v)$ holds witnessed by $v'$ and $S_{\varphi}(u,v)$ holds witnessed by $v''$.  Then we have $v \equiv^{L} v' \equiv^{L} v''$, $u \ind^{K} v'$, $u \ind^{K} v''$, $\varphi(x;u) \wedge \varphi(x;v')$ does not Kim-divide over $\emptyset$, and $\varphi(x;u) \wedge \varphi(x;v'')$ Kim-divides over $\emptyset$.  Choose $c'$ realizing $\varphi(x;u) \wedge \varphi(x;v')$ with $c' \ind^{K} u,v'$, and pick $c''$ so that $c'v' \equiv^{L} c''v''$.  Then, in particular, we have $c' \equiv^{L} c''$, $c' \ind^{K} u$, $c'' \ind^{K} v''$, and $u \ind^{K} v''$ so, by the independence theorem, there is $c$ with $c \models \mathrm{Lstp}(c'/u) \cup \mathrm{Lstp}(c/v'')$ and $c \ind^{K} u,v''$.  Since $c$ realizes $\varphi(x;u) \wedge \varphi(x;v'')$, we obtain a contradiction.  Together with the first part, this establishes the claim. \qed

By the claim and compactness, there is a formula $\sigma(x,y) = \sigma_{\varphi}(x,y)$ such that $p(x) \cup p(y) \vdash R_{\varphi}(x,y)$ if and only if  $\sigma(x,y)$. Moreover, it is clear from the definitions that $R_{\varphi}(x,y) \wedge r(y,y')$ implies $R_{\varphi}(x,y')$.  Again by compactness, there is a formula $\delta(x) \in p(x)$ such that 
\begin{itemize}
\item[(*)] $\delta(x) \wedge \delta(y) \wedge \sigma(x,y) \wedge \bigwedge r(x,z) \models \sigma(x,z).$
\end{itemize}
It follows from $(*)$ and the symmetry of $r$ that, for any $z \models \delta$, we have $\models \sigma(z,x) \leftrightarrow \sigma(z,y)$, for all $x,y \models \delta$.  Therefore, we obtain a definable equivalence relation $E_{\varphi}(x,y)$ as follows:
$$
E_{\varphi}(x,y) \equiv \left[ \neg \delta(x) \wedge \neg \delta(y) \right] \vee \left[ \delta(x) \wedge \delta(y) \wedge (\forall z) \left(\delta(z) \to (\sigma(z,x) \leftrightarrow \sigma(z,y) \right)\right].
$$
To conclude, we establish the following:

\textbf{Claim 2}: The partial type $p(x) \cup p(y)$ implies $r(x,y)$ holds if and only if $\bigwedge_{\varphi \in r} E_{\varphi}(x,y)$ holds.   

\emph{Proof of Claim}:  First, we will show $r(x,y)$ entails $E_{\varphi}(x,y)$ for all $\varphi \in r$.  Clearly $r(x,y)$ implies $\delta(x) \leftrightarrow \delta(y)$.  Moreover, as noted above, if $r(x,y)$ and $\delta(x) \wedge \delta(y)$ both hold, then for any $z \models \delta$ we have $\sigma(z,x) \leftrightarrow \sigma(z,y)$ by (*) above and the symmetry of $r(x,y)$.  Hence $E_{\varphi}(x,y)$ holds.  

Secondly, we will show that if $(a,b)\models p(x) \cup p(y)$ does not realize $r(x,y)$, then $\neg \bigwedge_{\varphi \in r} E_{\varphi}(a,b)$.  Choose $\psi(x,y) \in r(x,y)$ such that $\neg \psi(a,b)$.  Because $r$ is an equivalence relation, there is some $\varphi(x,y) \in r(x,y)$ such that 
$$
(\exists x,x') \left( \varphi(x,y) \wedge \varphi(x,x') \wedge \varphi(x',z) \right) \vdash \psi(y,z).
$$
Then if $E_{\varphi}(a,b)$ holds, then, because $R_{\varphi}(a,a)$ holds, we have $\sigma_{\varphi}(a,a)$ and $\sigma_{\varphi}(a,b)$ and therefore $R_{\varphi}(a,b)$.  By the definition of $R_{\varphi}(a,b)$, this implies there is some $c \equiv^{L} b$ such that $\varphi(x,a) \wedge \varphi(x,c)$ is consistent, and therefore $\models \psi(a,b)$, a contradiction.  This concludes the proof of the claim, and hence the theorem.   
\end{proof}

\section{Rank} \label{rank section}

In this section, we introduce a family of ranks, suitable for the study of NSOP$_{1}$ theories.  This makes critical use of witnessing over arbitrary sets and provides a clear context in which working over arbitrary sets greatly simplifies the situation.  Our definition is close to the definition of $D$-rank familiar from simple theories, but we are required to add a new parameter in the rank, which keeps track of the type of the parameters that appear in instances of Kim-dividing.  

\begin{defn} \label{rank def}
Suppose $q(y) \in S(B)$, $\Delta(x;y)$ is a finite set of $L(B)$-formulas, and $k < \omega$.  Then for any set of formulas $\pi(x)$ over $\mathbb{M}$, we define $D_{1}(\pi,\Delta,k,q) \geq 0$ if $\pi$ is consistent, and $D_{1}(\pi,\Delta,k,q) \geq n+1$ if there is a sequence $I = \langle c_{i} : i < \omega \rangle$ such that the following conditions hold:
\begin{enumerate}
\item The sequence $I$ is an $\ind^{K}$-Morley sequence over $B$ with $c_{i} \models q$.
\item The sequence $I$ is indiscernible over $\mathrm{dom}(\pi)B$ (and hence $\mathrm{dom}(\pi) \ind^{K}_{B} c_{i}$ for all $i$).
\item We have $\{\varphi(x;c_{i}) : i < \omega\}$ is $k$-inconsistent for some formula $\varphi(x;y) \in \Delta$.
\item We have $D_{1}(\pi \cup \{\varphi(x;c_{i})\},\Delta,k,q) \geq n$ for all $i < \omega$.  
\end{enumerate}
We define $D_{1}(\pi,\Delta,k,q) = n$ if $n$ is least such that $D_{1}(\pi,\Delta,k,q) \geq n$ and also $D_{1}(\pi,\Delta,k,q) \not\geq n+1$.  We say $D_{1}(\pi,\Delta, k,q) = \infty$ if $D_{1}(\pi,\Delta,k,q) \geq n$ for all $n < \omega$.  
\end{defn}

\begin{lem} \label{basic inequalities}
Suppose $q(y) \in S(B)$, $\Delta(x;y)$ is a finite set of $L(B)$-formulas, and $k <\omega$.  Then we have the following.  
\begin{enumerate}
\item For all $\sigma \in \mathrm{Aut}(\mathbb{M}/B)$, 
$$
D_{1}(\sigma(\pi),\Delta, k, q) = D_{1}(\pi,\Delta,k,q). 
$$
\item  If $\pi$ and $\pi'$ are partial types over $\mathbb{M}$ such that $\pi(x) \vdash \pi'(x)$, then
$$
D_{1}(\pi,\Delta,k,q) \leq D_{1}(\pi',\Delta,k,q).
$$
\item If $n \geq m$, then $D_{1}(\pi, \Delta,k,q) \geq n$ implies $D_{1}(\pi,\Delta,k,q) \geq m$.  
\item If $\psi_{0}(x),\ldots, \psi_{m-1}(x)$ are formulas over $\mathbb{M}$, then
$$
D_{1}\left(\pi \cup \left\{ \bigvee_{j < m} \psi_{i}(x)\right\},\Delta,k,q\right) = \max_{j < m} D_{1}(\pi \cup \{\psi_{j}(x)\},\Delta,k,q).
$$
\end{enumerate}
\end{lem}

\begin{proof}
(1) is clear.    

(2) By induction on $n$, we will show 
$$
D_{1}(\pi,\Delta,k,q) \geq n \implies D_{1}(\pi',\Delta,k,q) \geq n.
$$
For $n = 0$ there is nothing to show.  Suppose the statement holds for $n$, and assume $D_{1}(\pi,\Delta,k,q) \geq n+1$.  Then there is a $\mathrm{dom}(\pi)B$-indiscernible sequence $I = \langle c_{i} : i < \omega \rangle$ which is additionally an $\ind^{K}$-Morley sequence over $B$ in $q$ and $\varphi(x;y) \in \Delta$ such that $\{\varphi(x;c_{i}) : i < \omega\}$ is $k$-inconsistent and $D_{1}(\pi \cup \{\varphi(x;c_{i})\},\Delta,k,q) \geq n$.  Let $I' = \langle c'_{i} : i < \omega \rangle$ be a $B\mathrm{dom}(\pi)\mathrm{dom}(\pi')$-indiscernible sequence locally based on $I$.  Clearly we have that $I'$ is an $\ind^{K}$-Morley sequence in $q$ over $B$ and $D_{1}(\pi \cup \{\varphi(x;c'_{i})\},\Delta,k,q) \geq n$ for all $i < \omega$, by (1).  As $\pi \cup \{\varphi(x;c'_{i})\} \vdash \pi' \cup \{\varphi(x;c'_{i})\}$, we have $D_{1}(\pi' \cup \{\varphi(x;c'_{i})\},\Delta,k,q) \geq n$ by induction, which implies $D_{1}(\pi',\Delta,k,q) \geq n+1$.  

(3) We will prove by induction on $l < \omega$ that $D_{1}(\pi,\delta,k,q) \geq n+l$ implies $D_{1}(\pi,\delta,k,q) \geq n$.  For $l = 0$, this is trivial.  Assume it has been shown for $l$ and that $D_{1}(\pi,\Delta,k,q) \geq n + l+1$. Then there is a sequence $I = \langle c_{i} : i < \omega \rangle$ satisfying the conditions of Definition \ref{rank def} such that $D_{1}(\pi \cup \{\varphi(x;c_{i})\}, \Delta, k,q) \geq n + l$ for all $i$.  By (2), this entails, in particular, that $D_{1}(\pi,\Delta, k,q) \geq n+l$ and hence $D_{1}(\pi,\Delta,k,q) \geq n$ by induction. 

(4)  By (2), we have 
$$
 \max_{j < m} D_{1}(\pi \cup \{\psi_{j}(x)\},\Delta,k,q) \leq D_{1}\left(\pi \cup \left\{ \bigvee_{j < m} \psi_{j}(x)\right\},\Delta,k,q\right).
$$
Hence, it suffices to show, for $n < \omega$, that 
$$
D_{1}\left(\pi \cup \left\{ \bigvee_{j < m} \psi_{j}(x)\right\},\Delta,k,q\right) \geq n \implies \max_{j < m} D_{1}(\pi \cup \{\psi_{j}(x)\},\Delta,k,q) \geq n.
$$
Let $C$ be the (finite) set of parameters appearing in the formulas $\psi_{0},\ldots, \psi_{n-1}$.  For $n = 0$, this is clear.  If $D_{1}\left(\pi \cup \left\{ \bigvee_{j < m} \psi_{j}(x)\right\},\Delta,k,q\right) \geq n+1$, then, as in (2), there is an $\ind^{K}$-Morley sequence $I = \langle c_{i} : i < \omega \rangle$ over $B$ in $q$ which is $\mathrm{dom}(\pi)BC$-indiscernible and $\varphi(x;y) \in \Delta$ such that $\{\varphi(x;c_{i}) : i < \omega\}$ is $k$-inconsistent with
$$
D_{1}\left( \pi \cup \left\{ \bigvee_{j < m} \psi_{j}(x) \right\} \cup \{\varphi(x;c_{i})\},\Delta,k,q\right) \geq n,
$$  
which implies $\max_{j<m} D_{1}(\pi \cup \{\psi_{j}(x), \varphi(x;c_{i})\},\Delta,k,q) \geq n$, by the induction hypothesis for $n$, for each $i$.  By the pigeonhole principle, we may assume this maximum witnessed by the same $j$ for all $i < \omega$.  This shows $\max_{j<m} D_{1}(\pi \cup \{\psi_{j}(x)\},\Delta,k,q) \geq n+1$.  
\end{proof}

\begin{rem}
Lemma \ref{basic inequalities}(2) implies, in particular, that if $\pi(x)$ and $\pi'(x)$ are equivalent then the ranks (with respect to a choice of $\Delta$, $k$, and $q$) will be the same, even if they have different domains. 
\end{rem}

\begin{lem} \label{tree witness}
Suppose $T$ is NSOP$_{1}$ with existence.  Suppose $q \in S(B)$, $\Delta(x;y)$ is a finite set of $L(B)$-formulas, and $k < \omega$. Then for all $n < \omega$, we have $D_{1}(\pi,\Delta,k,q) \geq n$ if and only if there are $(c_{\eta})_{\eta \in \omega^{\leq n} \setminus \{\emptyset\}}$ and $\varphi_{i}(x;y) \in \Delta$ for $i < n$ satisfying the following:
\begin{enumerate}[(a)]
\item For all $\eta \in \omega^{n}$, 
$$
\pi(x) \cup \{\varphi_{i}(x;c_{\eta |(i+1)}) : i < n\}
$$
is consistent.
\item For all $\eta \in \omega^{< n}$, $\{\varphi_{l(\eta)}(x;c_{\eta \frown \langle i \rangle}) : i < \omega\}$ is $k$-inconsistent.
\item For all $\eta \in \omega^{< n}$, $\langle c_{\eta \frown \langle i \rangle} : i < \omega \rangle$ is an $\ind^{K}$-Morley sequence over $B$ in $q$.
\item The tree $(c_{\eta})_{\eta \in \omega^{\leq n} \setminus \{\emptyset\}}$ is $s$-indiscernible over $B\mathrm{dom}(\pi)$.  
\end{enumerate}
\end{lem}

\begin{proof}
For the case of $n = 0$, (a) is satisfied if and only if $\pi$ is consistent and (b)\textemdash(d) hold trivially, which gives the desired equivalence. 


Now assume, for a given $n$, that $D_{1}(\pi,\Delta,k,q) \geq n$ if and only if there are $(c_{\eta})_{\eta \in \omega^{\leq n} \setminus \{\emptyset\}}$ and $\varphi_{i}(x;y) \in \Delta$ for $i < n$ such that (a)\textemdash(d) hold. First, suppose $D_{1}(\pi,\Delta,k,q) \geq n+1$.  Then we can find a $\mathrm{dom}(\pi)B$-indiscernible sequence $I = \langle c_{i} : i < \omega \rangle$ which is also $\ind^{K}$-Morley over $B$ in $q$ and $\varphi(x;y) \in \Delta$ such that $\{\varphi(x;c_{i}) : i < \omega\}$ is $k$-inconsistent and $D_{1}(\pi \cup \{\varphi(x;c_{i})\},\Delta,k,q) \geq n$ for all $i < \omega$.  By induction, for each $i < \omega$, there is a tree $(c_{i,\eta})_{\eta \in \omega^{\leq n} \setminus \{\emptyset\}}$ and sequence of formulas $(\varphi_{i,j}(x;y))_{1 \leq j \leq n}$ from $\Delta$ satisfying (a)\textemdash(d), with $\pi$ replaced by $\pi \cup \{\varphi(x;c_{i})\}$.  As $\Delta$ is a finite set of formulas, we may, by the pigeonhole principle, assume that there are $\varphi_{j} \in \Delta$ such that $\varphi_{i,j}(x;y) = \varphi_{j}(x;y)$ for all $i < \omega$ and $1 \leq j \leq n$. Now define a tree $(c'_{\eta})_{\eta \in \omega^{\leq n+1} \setminus \{\emptyset\}}$ such that $c'_{\langle i \rangle} = c_{i}$ and $c'_{\langle i \rangle \frown \eta} = c_{i,\eta}$ for all $i < \omega$ and $\eta \in \omega^{\leq n} \setminus \{\emptyset\}$.  Define $\varphi_{0}(x;y) = \varphi(x;y)$.  Let $(c_{\eta})_{\eta \in \omega^{\leq n+1} \setminus \{\emptyset\}}$ be a tree that is $s$-indiscernible tree over $\mathrm{dom}(\pi)B$ and locally based on $(c'_{\eta})_{\eta \in \omega^{\leq n+1} \setminus \{\emptyset\}}$ over $B$. Note that for all $\eta \in \omega^{<n+1}$, the sequence $\langle c_{\eta \frown \langle i \rangle} : i < \omega \rangle$ has the same type over $B$ as a Morley sequence in $q$ and hence is a Morley sequence in $q$.  It is clear that $(c_{\eta})_{\eta \in \omega^{\leq n+1} \setminus \{\emptyset\}}$ and $(\varphi_{j}(x;y))_{j < n+1}$ satisfy the requirements.  

Conversely, given $(c_{\eta})_{\eta \in \omega^{\leq n+1} \setminus \{\emptyset\}}$ and $(\varphi_{j}(x;y))_{j < n+1}$ satisfying (a)\textemdash(d), we observe by induction that the tree $(c_{\langle i \rangle \frown \eta})_{\eta \in \omega^{\leq n} \setminus \{\emptyset\}}$ witnesses 
$$
D_{1}(\pi \cup \{\varphi_{0}(x;c_{\langle i \rangle})\},\Delta,k,q) \geq n,
$$ 
for all $i < \omega$.  As the sequence $\langle c_{\langle i \rangle} : i < \omega \rangle$ is $\ind^{K}$-Morley over $B$ and $\mathrm{dom}(\pi)B$-indiscernible, and $\{\varphi_{0}(x;c_{\langle i \rangle}) : i < \omega \}$ is $k$-inconsistent, it follows that
$$
D_{1}(\pi,\Delta,k,q) \geq n+1,$$ 
completing the proof.  
\end{proof}

\begin{cor} \label{finiteness}
If $T$ is NSOP$_{1}$ with existence, then given $q(y) \in S(B)$, a finite set of $L(B)$-formulas $\Delta(x;y)$, and $k < \omega$, there is some $n < \omega$ such that $D_{1}(x = x, \Delta,k,q) = n$.  
\end{cor}

\begin{proof}
Suppose towards contradiction that there are $q$, $\Delta$, and $k < \omega$ such that $D(x=x,\Delta, k,q) > n$ for all $n < \omega$.  By Lemma \ref{tree witness}, Lemma \ref{type definability of kim-independence}, compactness, and the pigeonhole principle (by the finiteness of $\Delta$), we can find a tree $(c_{\eta})_{\eta \in \omega^{<\omega} \setminus \{\emptyset\}}$ and $\varphi(x;y) \in \Delta$ satisfying the following:
\begin{enumerate}[(a)]
\item For all $\eta \in \omega^{\omega}$, 
$$
\{\varphi(x;c_{\eta | (i+1)}) : i < \omega\}
$$
is consistent.
\item For all $\eta \in \omega^{< \omega }$, $\{\varphi(x;c_{\eta \frown \langle i \rangle}) : i < \omega\}$ is $k$-inconsistent.
\item For all $\eta \in \omega^{< \omega}$, $\langle c_{\eta \frown \langle i \rangle} : i < \omega \rangle$ is an $\ind^{K}$-Morley sequence over $B$ in $q$.
\item The tree $(c_{\eta})_{\eta \in \omega^{<\omega} \setminus \{\emptyset\}}$ is $s$-indiscernible over $B$.  
\end{enumerate}
Note that, by $s$-indiscernibility, we have $\langle c_{0^{n} \frown \langle i \rangle} : i < \omega \rangle$ is an $\ind^{K}$-Morley sequence over $B$ which is $Bc_{0\frown 0^{< n}}$-indiscernible, for all $n < \omega$.  By witnessing, it follows that $c_{0 \frown 0^{<n}} \ind^{K}_{B} c_{0^{n+1}}$ for all $n$.  Let $\langle d_{i} : i < \omega \rangle$ be a $B$-indiscernible sequence locally based on $\langle c_{0^{n}} : 1 \leq n < \omega \rangle$ over $B$.  As each $c_{0^{n}} \models q$, for $1 \leq n < \omega$, we have that $d_{<n} \ind^{K}_{B} d_{n}$ for all $n$ as well, and therefore $\langle d_{i} : i < \omega \rangle$ is an $\ind^{K}$-Morley sequence over $B$ in $q'$.  Moreover, since $\{\varphi(x;c_{0^{n}}) : 1 \leq n < \omega\}$ is consistent, we know $\{\varphi(x;d_{n}) : n < \omega\}$ is consistent.  However, we know $\varphi(x;c_{0^{n}})$ Kim-divides over $B$ and hence $\varphi(x;d_{n})$ Kim-divides over $B$.  This contradicts witnessing (Theorem 2.5).  
\end{proof}

\begin{rem}
Corollary \ref{finiteness} has a converse:  if $T$ has SOP$_{1}$, then for some $B$, there is a $q(y) \in S(B)$, a finite set of $L(B)$-formulas $\Delta(x;y)$, and $k < \omega$ such that $D_{1}(x = x, \Delta,k,q) = \infty$.  One way to see this is to note that by \cite[Corollary 3.7]{dobrowolski2019independence} Kim's Lemma for non-forking Morley sequences fails over some set $B$ in any theory with SOP$_{1}$ (this doesn't use existence).  Concretely, this means there are Morley sequences $I = \langle a_{i} : i < \omega \rangle$ and $J = \langle b_{i} : i < \omega \rangle$ over $B$ with $a_{0} = b_{0}$ and a formula $\varphi(x;y)$ such that $\{\varphi(x;a_{i}) : i < \omega\}$ is $k$-inconsistent and $\{\varphi(x;b_{i}) : i < \omega\}$ is consistent. Then, by \cite[Lemma 3.4]{dobrowolski2019independence}, this implies there is a tree $(c_{\eta})_{\eta \in \omega^{<\omega}}$ satisfying the following properties:
\begin{enumerate}
\item For all $\eta \in \omega^{<\omega}$, $(c_{\eta \frown \langle i \rangle})_{i < \omega} \equiv_{B} I$.  
\item For all $\eta \in \omega^{<\omega}$, $(c_{\eta},c_{\eta | l(\eta) - 1}, \ldots, c_{\emptyset}) \equiv_{B} (b_{0},b_{1},\ldots, b_{l(\eta)})$.  
\item $(c_{\eta})_{\eta \in \omega^{<\omega}}$ is $s$-indiscernible over $B$. 
\end{enumerate}
Then, by Lemma \ref{tree witness}, for $q = \text{tp}(a_{0}/B)$, we have $D_{1}(x = x, \Delta, k,q) \geq n$ for all $n$. 
\end{rem}

\begin{lem} \label{finite or complete}
Suppose $T$ is NSOP$_{1}$ with existence.  Suppose $q(y) \in S(B)$, $\Delta(x,y)$ is a finite set of $L(B)$-formulas, and $k < \omega$.  
\begin{enumerate}
\item For any partial type $p$, there is a finite $r \subseteq p$ such that 
$$
D_{1}(p,\Delta,k,q) = D_{1}(r,\Delta,k,q).
$$
\item For any small set $A \subseteq \mathbb{M}$ and any partial type $p$ over $A$, there is $p' \in S(A)$ extending $p$ such that 
$$
D_{1}(p,\Delta,k,q) = D_{1}(p,\Delta,k,q).
$$
\end{enumerate}
\end{lem}

\begin{proof}
(1)  By Lemma \ref{basic inequalities}(1), if $r$ is a subtype of $p$ then $D_{1}(r,\Delta,k,q) \geq D_{1}(p,\Delta,k,q)$ so it suffices to find a finite $r \subseteq p$ with $D_{1}(r,\Delta,k,q) \leq D_{1}(p,\Delta,k,q)$.  Suppose $D_{1}(p,\Delta,k,q) = n$.  Consider, for each sequence $\overline{s} = (\varphi_{i}(x;y))_{i < n+1}$ of formulas from $\Delta$, the set of formulas $\Gamma_{\overline{s}}(x,(z_{\eta})_{\eta \in \omega^{\leq n+1} \setminus \{\emptyset\}})$ over $\mathrm{dom}(p)B$ expressing the following:
\begin{enumerate}
\item For all $\eta \in \omega^{n+1}$, 
$$
p(x) \cup \{\varphi_{i}(x;z_{\eta |(i+1)}) : i < n+1\}
$$
is consistent.
\item For all $\eta \in \omega^{< n+1}$, $\{\varphi_{l(\eta)}(x;z_{\eta \frown \langle i \rangle}) : i < \omega\}$ is $k$-inconsistent.
\item For all $\eta \in \omega^{< n+1}$, $\langle z_{\eta \frown \langle i \rangle} : i < \omega \rangle$ is an $\ind^{K}$-Morley sequence over $B$ in $q$ (possible by Lemma \ref{type definability of kim-independence}(2)).
\item The tree $(z_{\eta})_{\eta \in \omega^{\leq n+1} \setminus \{\emptyset\}}$ is $s$-indiscernible over $B\mathrm{dom}(p)$.  
\end{enumerate}
By Lemma \ref{tree witness} and the fact that $D_{1}(p,\Delta,k,q) = n < n+1$, we know that $\Gamma_{\overline{s}}$ is inconsistent.  By compactness, there is some finite $r_{\overline{s}}(x) \subseteq p(x)$ such that, replacing $p(x)$ with $r_{\overline{s}}(x)$ in (1), the formulas remain inconsistent.  Let $r$ be the union of $r_{\overline{s}}$ as $\overline{s}$ ranges over all length $n+1$ sequences of formulas of $\Delta$.  As $\Delta$ is finite, this is a finite set, so $r$ is a finite extension of each $r_{\overline{s}}$ and a subtype of $p$.  Then by Lemma \ref{tree witness} again, $D_{1}(r,\Delta,k,q) \leq n = D_{1}(p,\Delta,k,q)$.  

(2)  Let $\Gamma(x)$ be defined as follows:
$$
\Gamma(x) =  \{ \neg \psi(x) \in L(A) : D_{1}(p(x) \cup \{\psi(x)\}, \Delta,k,q) < D_{1}(p,\Delta,k,q)\}.
$$
If $p(x) \cup \Gamma(x)$ is inconsistent, then by compactness, there are $\psi_{0},\ldots, \psi_{n-1} \in \Gamma$ such that $p(x) \vdash \bigvee_{j < n} \psi_{j}$.  By Lemma \ref{basic inequalities}(1) and (2), this gives
\begin{eqnarray*}
D_{1}(p,\Delta,k,q) &=& D_{1} \left(p \cup \left\{\bigvee_{j < n} \psi_{j}(x)\right\},\Delta,k,q\right) \\
&=& \max_{j} D_{1}(p \cup \{\psi_{j}(x)\},\Delta,k,q) \\
&<& D_{1}(p,\Delta,k,q),
\end{eqnarray*}
a contradiction.  Therefore, we can choose a complete type $p' \in S(A)$ extending $p(x) \cup \Gamma(x)$.  By (1), if $D_{1}(p',\Delta,k,q) < D_{1}(p,\varphi,k,q)$, then there is a formula $\psi(x) \in p'$ such that $D_{1}(p \cup \{\psi(x)\}, \Delta,k,q) < D_{1}(p,\Delta,k,q)$ but this is impossible by the definition of $\Gamma$.  Therefore $p'$ is as desired.  
\end{proof}

\begin{thm}  \label{rank drop}
Assume $T$ is NSOP$_{1}$ with existence. 
\begin{enumerate}
\item Suppose $\pi$ is a partial type over $B$ and $\pi \subseteq \pi'$. If $\pi'$ Kim-divides over $B$, witnessed by the formula $\varphi(x;c_{0})$, and $I = \langle c_{i} : i < \omega \rangle$ is an $\ind^{K}$-Morley sequence over $B$ in $q$ such that $\{\varphi(x;c_{i}) : i < \omega\}$ is $k$-inconsistent, then we have 
$$
D_{1}(\pi',\varphi,k,q) < D_{1}(\pi,\varphi,k,q).
$$
\item If $T$ is simple, $a \ind_{BD} C$, then for any $q(y) \in S(B)$, any finite set of $L(B)$-formulas $\Delta(x,y)$, and any $k < \omega$, we have 
$$
D_{1}(p', \Delta,k,q) = D_{1}(p,\Delta,k,q),
$$
where $p' = \mathrm{tp}(a/BDC)$ and $p = \mathrm{tp}(a/BD)$.  
\end{enumerate}
\end{thm}

\begin{proof}
Suppose we are given $\pi \subseteq \pi'$, $\varphi$, $k$, $I$, and $q$ as in the statement.  We claim that $D_{1}(\pi \cup \{\varphi(x;c_{0})\},\varphi,k,q) < D_{1}(\pi,\varphi,k,q)$.  If not, then, again by Lemma \ref{basic inequalities}(1), we have 
$$
n := D_{1}(\pi \cup \{\varphi(x;c_{0})\},\varphi,k,q) = D_{1}(\pi,\varphi,k,q),
$$ 
and therefore, by $B$-indiscernibility, $D_{1}(\pi \cup \{\varphi(x;c_{i})\},\varphi,k,q) = n$ for all $i$.  This implies, by definition of the rank, that $D_{1}(\pi,\varphi,k,q) \geq n+1$, a contradiction.  Therefore, since $\varphi(x;c_{0}) \in \pi'$, we have
$$
D_{1}(\pi',\varphi,k,q) \leq D_{1}(\pi \cup \{\varphi(x;c_{0})\},\varphi,k,q) < D_{1}(\pi,\varphi,k,q),
$$
which proves (1).  

Now we prove (2).  As we are working in a simple theory, Kim-dividing and forking coincide by Kim's lemma \cite{kim1998forking}.  Fix an arbitrary $q \in S(B)$, finite set of $L(B)$-formulas $\Delta(x,y)$ and $k < \omega$.  By Lemma \ref{basic inequalities}(2), we have 
$$
D_{1}(p', \Delta,k,q) \leq D_{1}(p,\Delta,k,q).
$$
Hence, it suffices to show, by induction on $n < \omega$, 
$$
D_{1}(p,\Delta,k,q) \geq n \implies D_{1}(p', \Delta,k,q) \geq n.  
$$
For $n = 0$, this is clear, so assume it holds for $n$ and suppose $D_{1}(p,\Delta,k,q) \geq n+1$.  Then there is a Morley sequence $I = \langle c_{i} : i < \omega \rangle$ over $B$ in $q$ which is $BD$-indiscernible such that, for some $\varphi(x,y) \in \Delta$, $\{\varphi(x;c_{i}) : i < \omega\}$ is $k$-inconsistent and $D_{1}(p \cup \{\varphi(x;c_{i}) \}, \Delta, k, q) \geq n$ for all $i < \omega$. Let $p_{*} = p_{*}(x;c_{0})$ be a completion of $p \cup \{\varphi(x;c_{0})\}$ over $BDc_{0}$ with
$$
D_{1}(p_{*},\Delta,k,q) = D_{1}(p \cup \{\varphi(x;c_{0})\},\Delta,k,q) \geq n,
$$
which is possible by Lemma \ref{finite or complete}(2).  Without loss of generality, $a \models p_{*}$.

By extension and an automorphism over $aBD$, we may assume $C \ind_{BD} ac_{0}$, and hence there is $I' = \langle c'_{i} : i < \omega \rangle \equiv_{BDc_{0}} I$ such that $I'$ is $BCD$-indiscernible.  Note that $I'$ is still Morley in $q$.  Moreover, by base monotonicity and symmetry, this gives $a \ind_{BDc'_{0}} C$.  Let $p_{*}' = \text{tp}(a/BCDc'_{0}) \supseteq p'$. By the inductive hypothesis, the fact that $D_{1}(p_{*},\Delta,k,q) \geq n$ gives 
$$
D_{1}(p_{*}',\Delta,k,q) \geq n.
$$
By Lemma \ref{basic inequalities}(1) and (2), we obtain 
$$
D_{1}(p' \cup \{\varphi(x;c'_{i})\}, \Delta,k,q) \geq n
$$
for all $i < \omega$, which allows us to conclude $D_{1}(p',\Delta,k,q) \geq n+1$, completing the proof.  \end{proof}
%
%

\begin{rem}
By witnessing (Theorem \ref{witnessing}), we know that if $\pi$ is a partial type over $B$ and $\pi'$ Kim-divides over $B$, then this will be witnessed by some formula $\varphi(x;c_{0})$ implied by $\pi'$ and an $\ind^{K}$-Morley sequence over $B$.  Therefore, Theorem \ref{rank drop}(1) implies that, if for all $\varphi(x;y) \in L(B)$, $q(y) \in S(B)$, and $k < \omega$, 
$$
D_{1}(\pi',\varphi,k,q) = D_{1}(\pi,\varphi,k,q),
$$
then $\pi'$ does not Kim-divide over $B$.  
\end{rem}

\begin{quest}
Does Theorem \ref{rank drop}(2) hold for $\ind^{K}$ in all NSOP$_{1}$ theories satisfying existence?  Evidently, the proof above makes use of base monotonicity, which is known to fail in all non-simple NSOP$_{1}$ theories.  
\end{quest}

\section{The Kim-Pillay theorem over arbitrary sets} \label{Kim-Pillay theorem}

The Kim-Pillay-style criterion for NSOP$_{1}$ of \cite{ArtemNick} proceeds, essentially, by first showing that any relation that satisfies axioms (1)\textemdash (5) over models in the axioms below must be weaker than coheir independence, in the sense that if $M \models T$ and $\text{tp}(a/Mb)$ is finitely satisfiable in $M$ then $a \ind_{M} b$.  Consequently, this proof does not adapt to arbitrary sets.  Instead, we relate any relation satisfying the axioms to Kim-independence directly, using a tree-induction, to prove the following theorem.

\begin{thm} \label{criterion}
Assume $T$ satisfies existence.  The theory $T$ is NSOP$_{1}$ if and only if there is an \(\text{Aut}(\mathbb{M})\)-invariant ternary relation \(\ind\) on small subsets of the monster \(\mathbb{M} \models T\) which satisfies the following properties, for an arbitrary set of parameters \(A\) and arbitrary tuples from $\mathbb{M}$.
\begin{enumerate}
\item Strong finite character: if \(a \nind_{A} b\), then there is a formula \(\varphi(x,b,m) \in \text{tp}(a/bA)\) such that for any \(a' \models \varphi(x,b,m)\), \(a' \nind_{A} b\). 
\item Existence and Extension:  $a \ind_{A} A$ always holds and, if $a \ind_{A} b$, then, for any $c$, there is $a' \equiv_{Ab} a$ such that $a' \ind_{A} bc$.
\item Monotonicity: \(aa' \ind_{A} bb'\) \(\implies\) \(a \ind_{A} b\).
\item Symmetry: \(a \ind_{A} b \iff b \ind_{A} a\).
\item The independence theorem: \(a \ind_{A} b\), \(a' \ind_{A} c\), \(b \ind_{A} c\) and $a \equiv^{L}_{A} a'$ implies there is $a''$ with $a'' \equiv_{Ab} a$, $a'' \equiv_{Ac} a'$ and $a'' \ind_{A} bc$.
\end{enumerate}
Moreover, any such relation $\ind$ satisfying (1)-(5) strengthens $\ind^{K}$, i.e.~$a \ind_{A} b$ implies $a \ind^{K}_{A} b$.  If, additionally, $\ind$ satisfies the following,
\begin{enumerate}
\setcounter{enumi}{5}
\item Transitivity:  if $a \ind_{A} b$ and $a \ind_{Ab} c$ then $a \ind_{A} bc$.
\item Local character:  if $\kappa \geq |T|^{+}$ is a regular cardinal, $\langle A_{i} : i < \kappa \rangle$ is an increasing continuous sequence of sets of size $ <\kappa$, $A_{\kappa} = \bigcup_{i < \kappa} A_{i}$ and $|A_{\kappa}| = \kappa$, then for any finite $d$,  there is some $\alpha < \kappa$ such that $d \ind_{A_{\alpha}} A_{\kappa}$.    
\end{enumerate}
then $\ind = \ind^{K}$.  
\end{thm}

\begin{proof}
We know that if $T$ is NSOP$_{1}$ with existence then $\ind^{K}$ satisfies (1)-(5) by Fact \ref{basic Kim-independence facts}, so we will show the other direction.

First, assume $\ind$ satisfies axioms (1)-(5), and we will show that $\ind$ strengthens $\ind^{K}$.  By \cite[Theorem 6.1]{ArtemNick}, the existence of such a relation over models entails that the theory is NSOP$_{1}$.  Towards contradiction, suppose there is some set of parameters $A$ and tuples $a,b$ such that $a \ind_{A} b$ but $a \nind^{K}_{A} b$, witnessed by the formula $\varphi(x;b) \in \text{tp}(a/Ab)$.  By induction on the ordinals $\alpha$, we will construct trees $(b^{\alpha}_{\eta})_{\eta \in \mathcal{T}_{\alpha}}$ satisfying the following conditions for all $\alpha$:
\begin{enumerate}[(a)]
\item For all $\eta \in \mathcal{T}_{\alpha}$, $b^{\alpha}_{\eta} \models \text{tp}(b/A)$.  
\item $(b^{\alpha}_{\eta})_{\eta \in \mathcal{T}_{\alpha}}$ is $s$-indiscernible and weakly spread out over $A$.  
\item For $\alpha$ successor, $b^{\alpha}_{\emptyset} \ind_{A} b^{\alpha}_{\vartriangleright \emptyset}$.
\item If $\beta < \alpha$, then $b^{\alpha}_{\iota_{\beta \alpha}(\eta)} = b^{\beta}_{\eta}$ for all $\eta \in \mathcal{T}_{\beta}$.  
\end{enumerate}
To begin, we may take $b^{0}_{\emptyset} = b$ and at limits we take unions.  Given $(b^{\alpha}_{\eta})_{\eta \in \mathcal{T}_{\alpha}}$, let $I = \langle (b^{\alpha}_{i,\eta})_{\eta \in \mathcal{T}_{\alpha}} : i < \omega \rangle$ be a mutually $s$-indiscernible Morley sequence over $A$ with $(b^{\alpha}_{0,\eta})_{\eta \in \mathcal{T}_{\alpha}} = (b^{\alpha}_{\eta})_{\eta \in \mathcal{T}_{\alpha}}$, which exists by Lemma \ref{Morley mutual}.  We may apply extension to find $b_{*} \equiv_{A} b$ such that $b_{*} \ind_{A} I$.  Define a tree $(c_{\eta})_{\eta \in \mathcal{T}_{\alpha+1}}$ by setting $c_{\emptyset} = b_{*}$ and $c_{\langle i \rangle \frown \eta} = b_{i,\eta}^{\alpha}$ for all $i < \omega$ and $\eta \in \mathcal{T}_{\alpha}$.  Finally, we may define $(b^{\alpha+1}_{\eta})_{\eta \in \mathcal{T}_{\alpha+1}}$ to be an $s$-indiscernible tree over $A$ locally based on $(c_{\eta})_{\eta \in \mathcal{T}_{\alpha+1}}$.  After moving by an automorphism, we can assume $b^{\alpha+1}_{\iota_{\alpha\alpha+1}(\eta)} = b^{\alpha}_{\eta}$ for all $\eta \in \mathcal{T}_{\alpha}$.  This completes the construction.  By strong finite character and invariance, our construction ensures (b) and (c) will be satisfied for $(b^{\alpha+1}_{\eta})_{\eta \in \mathcal{T}_{\alpha+1}}$.  

Applying Erd\H{o}s-Rado, we obtain a weak Morley tree $(b_{\eta})_{\eta \in \mathcal{T}_{\omega}}$ over $A$ such that $b_{\zeta_{\alpha}}  \ind_{A} b_{\vartriangleright \zeta_{\alpha}}$ for all $\alpha < \omega$.  In particular, $(b_{\zeta_{\alpha}})_{\alpha < \omega}$ is an $\ind$-Morley sequence over $A$.  Define $\nu_{\alpha} = \zeta_{\alpha+1} \frown \langle 1 \rangle$ for all $\alpha < \omega$.  Because the tree is weakly spread out over $A$, we have for all $\alpha < \omega$, $b_{\unrhd \zeta_{\alpha+1} \frown \langle 1 \rangle} \ind^{K}_{A} b_{\unrhd \zeta_{\alpha + 1} \frown 0}$ and hence $b_{\nu_{\alpha}} \ind^{K}_{A} b_{\nu_{<\alpha}}$ since $(b_{\nu_{\beta}})_{\beta < \alpha}$ was enumerated in $b_{\unrhd \zeta_{\alpha+1} \frown 0}$.  Since the tree is a weak Morley tree, we have that both $(b_{\zeta_{\alpha}})_{\alpha < \omega}$ and $(b_{\nu_{\alpha}})_{\alpha < \omega}$ are $A$-indiscernible.

As $b \equiv_{A} b_{\zeta_{0}}$, there is $a_{0}$ such that $a_{0}b_{\zeta_{0}} \equiv_{A} ab$.  As $(b_{\zeta_{\alpha}})_{\alpha < \omega}$ is $A$-indiscernible, for each $\alpha > 0$, there is $\sigma_{\alpha} \in \mathrm{Autf}(\mathbb{M}/A)$ such that $\sigma_{\alpha}(b_{\zeta_{0}}) = b_{\zeta_{\alpha}}$.  Setting $a_{\alpha} = \sigma_{\alpha}(a_{0})$, we have $a_{\alpha}b_{\zeta_{\alpha}} \equiv^{L}_{A} a_{0}b_{\zeta_{0}}$. By the independence theorem for $\ind$, we may find some $a_{*}$ such that $a_{*} b_{\zeta_{\alpha}} \equiv_{A} ab$ for all $\alpha < \omega$.  In particular, this implies $\{\varphi(x;b_{\zeta_{\alpha}}) : \alpha < \omega\}$ is consistent.  However, since $\varphi(x;b)$ Kim-divides over $A$, $\{\varphi(x;b_{\nu_{\alpha}}) : \alpha < \omega\}$ is $k$-inconsistent for some $k$.  The $s$-indiscernibility of the tree implies that $b_{\zeta_{\alpha}} \equiv_{A b_{\zeta_{>\alpha}}b_{\nu_{>\alpha}}} b_{\nu_{\alpha}}$ for all $\alpha$, so by compactness, we have shown $\varphi$ has SOP$_{1}$.  This contradiction shows that $\ind$ strengthens $\ind^{K}$.

Secondly, assume additionally that $\ind$ satisfies (6) and (7).  The proof of Theorem \ref{witnessing} shows that (6) and (7) imply \emph{witnessing}:  if $I = (b_{i})_{i < \omega}$ is an $A$-indiscernible sequence with $b_{0} = b$ satisfying $b_{i} \ind_{A} b_{<i}$, then whenever $a \nind_{A} b$, there is $\varphi(x;c,b) \in \text{tp}(a/Ab)$ such that $\{\varphi(x;c,b_{i}) : i < \omega\}$ is inconsistent. 

Suppose that $a \ind^{K}_{A} b$ and, by extension and Erd\H{o}s-Rado, find an $\ind$-Morley sequence $I = \langle b_{i} : i < \omega \rangle$ over $A$ with $b_{0} =b$.  By the remarks above, we know that $I$ is, in particular, an $\ind^{K}$-Morley sequence over $A$, so there is $a' \equiv_{Ab} a$ such that $I$ is $Aa$-indiscernible (using $a \ind^{K}_{A} b)$.  By witnessing, this entails $a \ind_{A} b$.  In other words, $\ind$ and $\ind^{K}$ coincide.  
\end{proof}

\begin{rem}
It is clear from the proof of Theorem \ref{criterion} that, in order to get the same conclusion, we can replace (6) and (7) with witnessing:  if $I = (b_{i})_{i < \omega}$ is an $A$-indiscernible sequence with $b_{0} = b$ satisfying $b_{i} \ind_{A} b_{<i}$, then whenever $a \nind_{A} b$, there is $\varphi(x;c,b) \in \text{tp}(a/Ab)$ such that $\{\varphi(x;c,b_{i}) : i < \omega\}$ is inconsistent. 

To do so gives a result that more closely resembles \cite[Theorem 6.11]{kaplan2019transitivity}, while the formulation of Theorem \ref{criterion} above is closer to the original Kim-Pillay theorem for simple theories \cite[Theorem 4.2]{kim1997simple} (see also \cite[Theorem 3.3.1]{kim2013simplicity}).  
\end{rem}

%

\bibliographystyle{plain}
\bibliography{ms.bib}{}

\end{document}